\documentclass[journal ]{new-aiaa}
\usepackage[utf8]{inputenc}
\usepackage{textcomp}

\usepackage{graphicx}
\usepackage{amsmath}
\usepackage[version=4]{mhchem}
\usepackage{siunitx}
\usepackage{longtable,tabularx}

\usepackage{multirow}
\usepackage{bm}
\usepackage{algorithm}
\usepackage{algorithmic}
\usepackage{amsthm}

\usepackage{mathrsfs}
\usepackage{bigstrut}
\usepackage{booktabs}
\usepackage{lineno}
\usepackage{threeparttable}
\newtheorem{lemma}{Lemma}
\newtheorem{proposition}{Proposition}
\usepackage{algorithm}
\usepackage{algorithmic}

\usepackage{amsthm}
\usepackage{color}

\usepackage{xcolor}
\usepackage{hyperref}
\newcounter{aqctr}
\newenvironment{author-query}
{\refstepcounter{aqctr}\par\vspace{\baselineskip}\noindent
\color{red}\textbf{Author Query/Comment AQ \arabic{aqctr}.}}
{\par\vspace{\baselineskip}\normalcolor}

\newcommand{\symbf}[1]{\bm{#1}}
\newcommand{\increment}{\Delta}
\usepackage{makecell}

\title{Global Optimality in Multi-Flyby Asteroid Trajectory Optimization: Theory and Application Techniques}
\author{Zhong Zhang\footnote{Postdoctoral Researcher, School of Aerospace Engineering; Department of Aerospace Science and Technology, zhong.zhang@polimi.it.} } 
\affil{Tsinghua University, 100084 Beijing, People's Republic of China; Politecnico di Milano, 20156 Milan, Italy}
\author{Xiang Guo \footnote{Lecturer, College of Aerospace Science and Engineering; State Key Laboratory of Space System Operation and Control, gx9645@126.com}} 
\affil{National University of Defense Technology, 410073 Changsha, People’s Republic of China}
\author{Di Wu\footnote{Associate Professor, School of Astronautics, wudi2025@buaa.edu.cn, Member AIAA.}}
\affil{Beihang University, Beijing, 100191, People’s Republic of China}
\author{Hexi Baoyin\footnote{Professor, School of Aerospace Engineering, baoyin@tsinghua.edu.cn. Senior Member AIAA}}
\author{Junfeng Li\footnote{Professor, School of Aerospace Engineering; lijunf@tsinghua.edu.cn. Senior Member AIAA (Corresponding Author)}}
\affil{Tsinghua University, 100084 Beijing, People's Republic of China}
\author{Francesco Topputo\footnote{Professor, Department of Aerospace Science and Technology, francesco.topputo@polimi.it, Senior Member AIAA}}
\affil{Politecnico di Milano, 20156 Milan, Italy}
\begin{document}
\maketitle
\vspace{-1em}   
\begin{abstract}
	Designing optimal trajectories for multi-flyby asteroid missions is scientifically critical but technically challenging due to nonlinear dynamics, intermediate constraints, and numerous local optima.
	This paper establishes a method that approaches global optimality for multi-flyby trajectory optimization under a given sequence. The original optimal control problem with interior-point equality constraints is transformed into a multi-stage decision formulation. This reformulation enables direct application of dynamic programming in lower dimensions, and follows Bellman’s principle of optimality. Moreover, the method provides a quantifiable bound on global optima errors introduced by discretization and approximation assumptions, thus ensuring a measure of confidence in the obtained solution.
	The method accommodates both impulsive and low-thrust maneuver schemes in rendezvous and flyby scenarios. Several computational techniques are introduced to enhance efficiency, including a specialized solution for bi-impulse cases and an adaptive step refinement strategy.
	The proposed method is validated through three problems: 1) an impulsive variant of the fourth Global Trajectory Optimization competition problem (GTOC4), 2) the GTOC11 problem, and 3) the original low-thrust GTOC4 problem. Each case demonstrates improvements in fuel consumption over the best-known trajectories. These results give evidence of the generality and effectiveness of the proposed method in global trajectory optimization.
\end{abstract}

\section*{Nomenclature}
 {\renewcommand\arraystretch{1.0}
  \noindent\begin{longtable*}{@{}l @{\quad=\quad} l@{}}
	  $a$ &  semimajor axis \\
	  $\bm d$ & decision variable\\
	  $\mathscr{D}$ & admissible state set\\
	  $e$ &   eccentricity \\
	  $g_0$ & sea-level standard acceleration due to gravity \\
	  $g_k$ & optimal cost for the $k$-th stage\\
	  $\tilde{g}_k$ & approximated optimal cost for the $k$-th stage\\
	  $i$ &  inclination angle \\
	  $I_{\rm sp}$ & specific impulse\\
	  $J$ &  objective function \\
	  $J^*$ &  optimal objective function value\\

	  $m$ & spacecraft mass\\
	  $m_{0}$ &  spacecraft initial mass \\
	  $m_{\rm fuel}$ &  propellant mass\\
	  $m_{f}$ & spacecraft final mass\\
	  $N$ & number of flyby bodies\\
	  $N_m$ & number of discrete masses\\
	  $N_s$ & total number of states\\
	  $N_t$ & number of discrete times\\
	  $N_v$ & number of discrete velocities\\

	  $P_0$ & original optimal control problem\\
	  $P_1$ & multi-stage decision problem\\
	  $P_2$ & discretized and approximated multi-stage decision problem\\
	  $P_3$ & stochastic multi-stage decision problem\\

	  $\bm R$ & position vector of flyby body\\
	  $\bm r$ & spacecraft position vector\\

	  $\bm s$ & state variable\\
	  $\mathscr{S}$ & admissible state set\\

	  $\bm u$ & control vector \\

	  $\bm V$ & velocity vector of flyby body\\
	  $\bm v$ & spacecraft velocity vector\\

	  $\Delta v$ & velocity increment\\
	  $\Delta v_{\rm leg}$ & velocity increment of each leg\\
	  $\varepsilon$ & prediction error\\
	  $\mu$ &  gravitational parameter \\
	  $\symbf{\mu}$ &  control policy\\
	  $\mathscr{U}$ &  admissible policy set\\
	  \multicolumn{2}{@{}l}{Subscripts}\\
	  $k$ & $k$-th flyby \\
  \end{longtable*}}

\section{Introduction}

Asteroid exploration has attracted much attention due to its scientific importance and potential for resource extraction. Only a few dozen among millions of known asteroids have been visited by spacecraft. Missions targeting the exploration of multiple asteroids have thus become a focal point due to their ability to achieve greater scientific impact with lower costs. NASA's Lucy mission is set to investigate several Trojan asteroids \cite{levison2021lucy}. During its encounter with asteroid 152830 Dinkinesh, Lucy made the unexpected discovery of a binary structure \cite{Levison2024}. Studies have shown that JAXA's DESTINY+ could extend its mission to include flybys of five additional asteroids \cite{Ozaki2022}, following its primary objective of exploring asteroid 3200 Phaethon. Other notable asteroid missions, such as ESA's Hera \cite{michelESAHeraMission2022,ferrariPreliminaryMissionProfile2021}, ESA's RAMSES\cite{morelliInitialTrajectoryAssessment2024}, JAXA's Hayabusa-2 \cite{watanabe2019hayabusa2} and NASA's OSIRIS-REx \cite{scheeres2019dynamic}, have consistently yielded groundbreaking findings, further driving enthusiasm and momentum in the field of asteroid exploration.

Optimizing multiple asteroid-linking trajectories is a complex and demanding task in interplanetary mission design. It involves scenarios with both multi-flyby and multi-rendezvous. The optimization process typically consists of two steps: first, identifying feasible asteroid exploration sequences through global optimization; second, refining trajectories using local optimization under a given sequence. The first step involves using global optimization algorithms to  determine a feasible sequence of asteroids \cite{Viavattene2022,DICARLO20182026,7850107}. In this phase, the selection of candidate asteroids prioritizes computational efficiency over model fidelity, aiming to rapidly narrow down the vast search space. The optimization problem is inherently a complex combinatorial challenge featuring nonlinear time-dependent dynamics. To tackle such problems, heuristic algorithms such as genetic algorithms \cite{sanchez2016asteroid, izzo2015evolving}, ant colony optimization \cite{Hong-Xin2018}, and tree search \cite{9623484, Izzo2019_chapter} are commonly employed. The second step involves solving a nonlinear optimal control problem to refine the trajectory for a predefined asteroid visit sequence. This step typically relies on either direct methods \cite{fahroo2002direct, Ozaki2017,Lantoine2012,Aziz2019,doi:10.2514/1.G008107} or indirect methods \cite{Russell2007,doi:10.2514/1.G000379,doi:10.2514/6.2004-5088}, which are chosen based on the specific requirements of the mission. This paper considers the second step, with a focus on designing fuel-optimal trajectories for a given flyby sequence.

Benchmark problems for multi-asteroid trajectories have also been developed as essential tools for evaluating optimization algorithms. Out of the 12 Global Trajectory Optimization Competitions (GTOCs) held to date\footnote{Izzo, D., “GTOC Website”  \url{https://sophia.estec.esa.int/gtoc\_portal/} [retrieved 1 December 2024]}, seven have involved asteroid-related missions \cite{li2018review,SHEN2023889}, all of which included multi-flyby or multi-rendezvous scenarios. These competitions highlight the inherent complexity of trajectory optimization for such missions. The availability of diverse benchmark problems provides researchers with a reliable framework for comparing optimization algorithms. This allows direct benchmarking against state-of-the-art methods without requiring detailed knowledge of their implementation. Moreover, it facilitates the straightforward validation of new algorithms by other researchers.

For the multi-flyby trajectory design problem, D'Amario et al. \cite{DAmario1981} investigated the direct method for impulsive multi-flyby problems and computed analytical derivatives. Subsequently, Olympio et al. \cite{Olympio:2011:Optimal} studied the indirect method for multi-flyby low-thrust trajectory optimization problems. Wang et al. \cite{wangAnalyticGradientsNormalized2024} derived analytical gradients for interior point constraints, which improved the efficiency of solving low-thrust optimal control problems.
Broadly speaking, the direct method discretizes the problem into a parameter optimization problem. Its primary advantage lies in its ability to handle many intermediate constraints. However, it suffers from limited control precision and can result in a high-dimensional optimization problem that remains computationally challenging. The indirect method, in contrast, addresses a low-dimensional nonlinear differential equation problem. However, it is highly sensitive to the initial guesses for the costate variables, making convergence difficult. With advancements in machine learning, neural networks have also been proposed to approximate solutions to these problems \cite{Viavattene2022, Ozaki2022} using training data derived from direct or indirect methods.

The methods above are typically limited to finding local optima. Previous studies have demonstrated the existence of numerous local optima, illustrating the inherent complexity of the problem \cite{Russell2007,doi:10.2514/1.G007138}. This finding highlights the critical importance of global optimization. Transforming the problem into a low-dimensional formulation and leveraging modern computational power to approximate exhaustive searches for a global optimum has become an important and challenging field of research. Zhang et al. \cite{Zhang2024} explored the application of dynamic programming to impulsive multi-rendezvous problems. Independently, Bellome et al. \cite{Bellome2023, Bellome2024} proposed a dynamic programming method for bi-impulse solutions in multi-flyby problems. Several critical challenges remain to be addressed, including the development of a more general and rigorous algorithmic proof of global optimality that unifies impulsive and low-thrust maneuvers, as well as multi-flyby and multi-rendezvous problems. Furthermore, previous studies have inevitably employed discretization to transform continuous variables, such as flyby epochs, into discrete counterparts, introducing errors. Establishing a theoretical bound for such errors is another significant theoretical challenge.

This paper aims to develop a rigorous and general global optimization method to address the challenges of multi-flyby trajectory optimization. The main contributions are summarized as follows.
\begin{itemize}
	\item On the theoretical side, a problem reformulation is established which integrates multi-stage decision-making with dynamic programming methods. This method rigorously proves global optimality for both impulsive and low-thrust maneuvers in rendezvous and flyby scenarios. Additionally, the relationship between the globally optimal solutions of the discretized approximation and the original continuous problem is thoroughly analyzed, resulting in rigorously derived error bounds for the first time.
	\item From a practical perspective, this paper introduces a specialized solution for pre-designed bi-impulse problems to significantly enhance computational efficiency. Furthermore, an adaptive step technique is proposed to accelerate the solution process. Although the adaptive step technique does not theoretically guarantee global optimality, it performs excellently in practical test cases.
	\item The proposed method is applied to three benchmark problems: the GTOC4 problem, the GTOC4 impulsive variant, and the GTOC11 problem. In all cases, the method achieves substantial fuel savings compared to existing optimal solutions, highlighting its robustness and effectiveness in both impulsive and low-thrust mission scenarios. The open-source code accompanying this work provides a valuable tool for researchers, enabling its use in future research or as a benchmark for future studies.
\end{itemize}

The remainder of this paper is organized as follows: Section \ref{sec:modeling} introduces the modeling of the multi-flyby problem. Section \ref{sec:proof} details the proof of global optimality and the derivation of error bounds. Section \ref{sec:analysis_discussion} discusses global optimality and error bounds, examines the connections between low-thrust and impulsive schemes, as well as flyby and rendezvous scenarios, and proposes techniques to improve algorithmic efficiency. Section \ref{sec:examples} illustrates the effectiveness of the proposed framework by solving three GTOC problems with optimized solutions. Finally, Sec.~\ref{Sec:conclusion} summarizes the findings and contributions of the paper.

\section{Modeling of the Multi-Flyby Problem}
\label{sec:modeling}
This section considers the general formulation for optimizing multi-flyby trajectories under a given flyby sequence. This starts with an optimal control problem that includes interior-point equality constraints and is transformed into a multi-stage decision problem. After discretizing and approximating the problem, the tractable problem solved in this paper is obtained. 

\subsection{Optimal Control Problem $P_0$ with Interior-Point Equality Constraints}

This section defines the multi-flyby trajectory optimization problem for a spacecraft, which is an optimal control problem with interior point equality constraints, hereinafter referred to as problem $P_0$:


\textcolor{black}{Assume the spacecraft has $N+1$ flyby events, starting from the departure body (indexed by 1) and ending at the final flyby body (indexed by $N+1$), with the mission time spanning from $t_0$ to $t_{\rm f}$. The use of $N+1$ flyby bodies facilitates the next formulation of the problem as an $N$-stage decision process, with each stage corresponding to a transfer between two consecutive flyby bodies.}

The objective of the optimization problem is to minimize fuel consumption:
\begin{equation}
	\min J = \left( {{m_0} - {m_f}} \right)\quad \text{or} \quad {m_{{\rm{fuel }}}}\quad \text{or} \quad  \sum \increment {v}
\end{equation}
The system's dynamic equations are given in the general case without specifying the spacecraft's dynamical environment, considering both impulsive and low-thrust maneuvers:
\begin{equation}
	\label{eq:dynamic_general}
	\dot{\symbf{x}}(t) = \symbf{f}\left( \symbf{x}(t), \symbf{u}(t), t \right), \quad t \in [t_0, t_{\rm f}]
\end{equation}
Specifically, for the test cases presented in Sec.~\ref{sec:examples}, continuous and impulsive thrust is considered in a two-body gravitational field.

For the low-thrust case, Eq.~(\ref{eq:dynamic_general}) can be explicitly expressed as:
\begin{equation}
	\label{eq:lowthrust_dynamic}
	\left\{
	\begin{aligned}
		\symbf{\dot{r}} & = \symbf{v}                                        \\
		\symbf{\dot{v}} & = -\mu \frac{\symbf{r}}{r^3} + \frac{\symbf{T}}{m} \\
		\dot{m}         & = -\frac{\|\symbf{T}\|}{I_{\text{sp}}g_0}
	\end{aligned}
	\right.
\end{equation}
where \(\symbf{r} = [x, y, z]^\mathrm{T}\) and \(\symbf{v} = [u, v, w]^{\mathrm{T}}\) represent the position and velocity vectors of the spacecraft, respectively.
The mass of the spacecraft is denoted by \(m\), and \(r = \sqrt{x^2 + y^2 + z^2}\) represents the distance from the spacecraft to the Sun.
The gravitational parameter of the Sun is \(\mu\).
The thrust vector is given by \(\symbf{T} = [T_x, T_y, T_z]^\mathrm{T}\), and its magnitude is \(T = \sqrt{T_x^2 + T_y^2 + T_z^2}\).
The specific impulse is \(I_{\text{sp}}\), and the standard gravitational acceleration is \(g_0\).
The magnitude of the thrust is constrained by $\|\symbf{T}\| \leq T_{\text{max}}$, and $\|.\|$ represents the 2-norm.

For the impulsive thrust case, the spacecraft motion between two successive impulses is governed by the two-body gravitational field, which corresponds to setting $\|\symbf{T}\| = 0$ in Eq.~(\ref{eq:lowthrust_dynamic}). Each impulsive maneuver is modeled as an instantaneous change in velocity and mass:
\begin{equation}
\left\{
\begin{aligned}
	\symbf{r}(t^+) & = \symbf{r}(t^-),                   \\
	\symbf{v}(t^+) & = \symbf{v}(t^-) + \Delta\symbf{v}, \\
	m(t^+) & = m(t^-) \exp\left(-\frac{\|\Delta \symbf{v}\|}{g_0 I_{\text{sp}}}\right),
\end{aligned}
\right.
\label{eq:impulse_dv}
\end{equation}
where the impulse vector is \(\Delta\symbf{v}\), and the moments immediately before and after the impulse are denoted by \(t^-\) and \(t^+\), respectively.




The constraints of problem $P_0$ are the flyby conditions, which require that the spacecraft's position coincides with the flyby target at each flyby epoch without imposing restrictions on the velocity:
\begin{align}
	r_x(t_k) & = R_{kx}(t_k)                                                  \\
	r_y(t_k) & = R_{ky}(t_k) \quad \text{for } k = 1, \cdots, N + 1           \\
	r_z(t_k) & = R_{kz}(t_k) 
	\\ t_k      & \leq t_{k + 1} \quad\quad \text{for } k = 1, \cdots, N.
\end{align}
Here, \(\symbf{R}_k=[R_{kx},R_{ky},R_{kz}]^{\rm T}\) represents the position of the flyby target, and \(t_k\) indicates the time of the \(k\)-th flyby event, with a total of \(N+1\) flyby events.

Problem $P_0$ is a general formulation of the multi-flyby trajectory optimization problem, which can be solved using either direct or indirect optimal control methods. The primary difficulty in solving this problem using classical methods lies in the flyby events being interior point equality constraints. Classical optimal control methods face significant challenges: direct methods often lead to a substantial increase in problem dimensionality, pseudospectral methods~\cite{fahroo2002direct} struggle to enforce interior-point constraints due to their reliance on orthogonal collocation points, and traditional indirect methods are hampered by the need for accurate initial guesses of the costates, and also the emergence of the other multipliers associated to the interior-point constraints. These challenges become increasingly problematic as the number of flyby events increases~\cite{Olympio:2011:Optimal}. 

\subsection{Multi-Stage Decision Problem $P_1$}


\textcolor{black}{
As previously mentioned, the complexity of the problem mainly arises from the interior-point constraints imposed by flyby events. 
This study leverages the flyby event constraints to reformulate the original problem $P_0$ as a two-level optimization framework. The inner level solves an optimal control problem without intermediate constraints, with the initial and final states held fixed; while the outer level is formulated as a multi-stage decision problem, aimed at determining the state variables at each stage to find globally optimal solution. 
For clarity, it is important to note that problem $P_1$ is not the final problem to be solved, but rather a virtual intermediate formulation that facilitates a more rigorous mathematical derivation. Some definitions are designed to facilitate later derivations using dynamic programming, and may differ from those commonly seen in direct and indirect methods of optimal control.}

To begin, we define the inner optimal control problem. The globally optimal transfer cost between the $k$-th and $(k+1)$-th flyby events is defined as $g_k(\symbf{s}_k, \symbf{d}_k)$. 
It is assumed that the value of $g_k$ is known, and a method to obtain this value will be presented in Sec.~\ref{sec:discretized_approximation}. $g_k(\symbf{s}_k, \symbf{d}_k)$ can be written as: 
\textcolor{black}{
\begin{align}
	\label{eq:inner_problem}
	 {g_k}({{\symbf{s}}_k},{{\symbf{d}}_k}) &= \min_{\symbf{u}(t)} \increment {v}, \quad k = 1,2, \cdot  \cdot  \cdot ,N                                  \\
	{\rm{s.t.}}  \quad  \symbf{s}_{k+1} &= \symbf{d}_k \notag  \\
	\dot{\symbf{x}}(t) &= \symbf{f}\left( \symbf{x}(t), \symbf{u}(t), t \right), \quad t \in [t_k, t_{k+1}]   \notag  \\
	{\symbf{r}(t_k)} &= \symbf{R}_k(t_k), \quad \symbf{r}(t_{k+1}) = \symbf{R}_{k+1}(t_{k+1}) \notag.	\\
	{\symbf{v}(t_k)} &=  \symbf{v}_k, \quad \symbf{v}(t_{k+1}) =  \symbf{v}_{k+1} \notag. \\
	{m(t_k)} &= m_k  \notag
\end{align}
}
Here, $\symbf{s}_k = [t_k; \symbf{v}_k; m_k]$ represents the state vector, and $\symbf{s}_k \in \mathscr{S}_k$ indicates that the state space is consistent with the possible states achievable in problem $P_0$.
The vector $\symbf{d}_k$ represents the decision variables, and $\mathscr{D}_k$ represents the decision space, i.e., $\symbf{d}_k \in \mathscr{D}_k(\symbf{s}_k)$. 
\textcolor{black}{The state transition equation is $\symbf{s}_{k+1} = \symbf{d}_k$. 
Note that the decision variables $\symbf{d}_k$ here refer to those of the outer $N$-stage decision problem, rather than the inner spacecraft control variable $\symbf{u}$. }
The subscript $k$ indicates the value at the $k$-th stage. 

The above optimization problem Eq.~\eqref{eq:inner_problem} is a typical two-point boundary value optimal control problem. Importantly, $g_k$ can be expressed as an explicit function of $\symbf{s}_k$ and $\symbf{d}_k$ independent of other state variables. \textcolor{black}{This characteristic reflects the Markov property of the multi-flyby problem, which states that only the velocity and mass of spacecraft at the flyby epoch influence subsequent trajectories.}

Thus, using $g_k(\symbf{s}_k, \symbf{d}_k)$ as the cost function for the $k$th stage, we construct the objective for the outer $N$-stage decision problem. The $N$-stage decision problem $P_1$ is defined as:
\textcolor{black}{
\begin{align}
	\label{eq:multi_stage_problem_p1}
	\min J                                         & =  \sum_{k = 1}^{N} g_k(\symbf{s}_k,\symbf{d}_k)                                 \\
	\text{s.t.} \quad g_k(\symbf{s}_k,\symbf{d}_k) & = \min \increment v, \quad k = 1,2, \ldots ,N \notag                                              \\
 \symbf{s}_{k+1} &= \symbf{d}_k \notag             
\end{align}
}

In Sec.~\ref{sec:optimal values are equal}, it will be rigorously proven that the globally optimal values of problems $P_1$ and $P_0$ are equal. Therefore, problem $P_0$ can be reduced to problem $P_1$. However, this reformulation does not offer a substantial improvement over directly solving problem $P_0$. First, introducing a two-level structure does not simplify the problem; rather, it increases its complexity. Second, solving Eq.~\eqref{eq:multi_stage_problem_p1} involves infinitely many combinations of $g_k(\symbf{s}_k, \symbf{d}_k)$ for different continuous state variables. Finally, solving the optimal control problem $g_k(\symbf{s}_k, \symbf{d}_k)$ is computationally intensive. Nonetheless, transforming the optimal control problem with interior-point equality constraints into a multi-stage decision problem offers an alternative framework for addressing such problems. Once the problem is discretized and approximated, this approach proves to be advantageous in facilitating numerical solution.

\subsection{Discretization and Approximation of the Multi-Stage Decision Problem $P_2$}
\label{sec:discretized_approximation}

For problem $P_1$, if we could skip the step of solving the optimal control problem and directly obtain $g_k(\symbf{s}_k, \symbf{d}_k)$, the problem-solving process would naturally be greatly simplified. Fortunately, some available analytical tools \cite{DICARLO20182026,7850107} and a wide range of emerging database methods \cite{ZHANG2023819} and machine learning techniques \cite{zhu2019fast,7850107,Ozaki2022,viavattene2022artificial,SONG201928,zhengModelIncrementalLearning2025,quExperienceReplayEnhances2025,yangDeepNeuralNetworksBased2025} provide the possibility to quickly predict $g_k(\symbf{s}_k, \symbf{d}_k)$.
In short, the accuracy and speed of predicting $g_k(\symbf{s}_k, \symbf{d}_k)$ are continuously improving with advancements in computational techniques.

Therefore, the values of $g_k$ can be obtained through approximations, denoted as $\tilde{g}_k$:
\begin{equation}
	\label{eq:error}
	\tilde{g}_k(\symbf{s}_k, \symbf{d}_k) = g_k(\symbf{s}_k, \symbf{d}_k) + \varepsilon_k = \increment {{\tilde v}_{\min }}, \quad k = 1,2, \ldots ,N
\end{equation}
where $\varepsilon_k$ represents the prediction error. 

Furthermore, due to computational limitations, it is stipulated that all state variables can only take discrete values, and the decision variables change accordingly. Thus, the state space $\mathscr{S}_k$ and decision space $\mathscr{D}_k$ for problem $P_2$ can only take certain discrete values of problem $P_1$. 
\textcolor{black}{As a result, the state transition equation becomes
$
	\symbf{s}_{k+1} = \symbf{d}_k + \symbf{\theta}_k,
$
where $\symbf{\theta}_k$ represents the rounding error introduced by discretization.}

The discretization and approximation of the $N$-stage decision problem $P_2$ is represented as:
\begin{align}
	\min J                                                            & =  \sum\limits_{k = 1}^{N} {{{\tilde g}_k}({{\symbf{s}}_k},{{\symbf{d}}_k})} \\
	{\rm{s.t.}} \quad {{\tilde g}_k}({{\symbf{s}}_k},{{\symbf{d}}_k}) & = \increment {{\tilde v}_{\min }}, \quad k = 1,2, \ldots ,N \notag                                    
	\\ 	\symbf{s}_{k+1} &= \symbf{d}_k + \symbf{\theta}_k. \notag 
\end{align}

Compared to problem $P_1$, problem $P_2$ eliminates the step of solving the optimal control problem and directly uses the approximated values ${{\tilde g}_k}({{\symbf{s}}_k},{{\symbf{d}}_k})$ while discretizing the state variables to meet computational performance requirements.

\section{Proof of Global Optimality and Error Bound Derivation}
\label{sec:proof}

This section addresses several theoretical issues. First, the question of whether the globally optimal solutions of problems \( P_0 \) and \( P_1 \) are identical is examined. Second, a method for obtaining the globally optimal value of problem \( P_2 \) is introduced and rigorously proven. Finally, the deviation between the solution of problem \( P_2 \) and the globally optimal solution of the original problem \( P_0 \)  is analyzed to assess the accuracy of the approximate solution. 

\subsection{Proof of Equivalence of Globally Optimal Values for Problems $P_0$ and $P_1$}

Let $\{\symbf{x}^*(t), \symbf{u}^*(t)\}$ be the optimal solution to problem $P_0$, with corresponding optimal cost denoted by $J^*$. By definition of problem $P_0$, it follows that
\begin{equation}
	J^* = \min_{\symbf{u}(t)} J = \min_{\symbf{u}(t)} \left( m_0 - m_f \right) \quad \text{or} \quad m_{\text{fuel}} \quad \text{or} \quad \sum \increment v
\end{equation}

The globally optimal cost for the $N$-stage decision problem $P_1$ is defined as
\begin{equation}
	G^\dagger = \sum_{k=1}^{N} g_k^\dagger = \min_{g_k} \left( \sum_{k=1}^{N} g_k \right)
\end{equation}
where $g_k^\dagger$ denotes the stage cost at stage $k$ corresponding to the optimal total cost $G^\dagger$.

To show that the global optimal values of problems $P_0$ and $P_1$ are equal, it suffices to prove that
\begin{equation}
	\label{eq:p0=p1}
	J^* = G^\dagger.
\end{equation}

\begin{proof}

According to the definition of problem $P_1$, any control strategy in $P_1$ is also feasible for problem $P_0$. Therefore, the solution space of $P_1$ is a subset of that of $P_0$. It follows that
\[
	J^* \le G^\dagger.
\]

To complete the proof, it remains to show that the following inequality holds:
\[
	J^* \ge G^\dagger =  \sum_{k=1}^{N} g_k^\dagger.
\]

Consider splitting the optimal solution $J^*$ of problem $P_0$ into $N$ segments, each transferring the state from $\symbf{s}_k^*$ to $\symbf{s}_{k+1}^*$ for $k = 0, 1, \dots, N-1$. Denote the cost of stage $k$ by $g_k^*$. Then, the total cost is
\[
	J^* = \sum_{k=1}^{N} g_k^*.
\]

By the definition of $g_k$, it holds that
\[
	g_k^*\left(\symbf{s}_k^*, \symbf{d}_k^*\right) \ge g_k\left(\symbf{s}_k^*, \symbf{d}_k^*\right), \quad k = 1, 2, \dots, N.
\]

Therefore,
\[
	J^* = \sum_{k=1}^{N} g_k^* \ge \sum_{k=1}^{N} g_k.
\]

Since $G^\dagger$ is the global optimal cost of problem $P_1$, we have
\[
	\sum_{k=1}^{N} g_k \ge G^\dagger.
\]

Hence,
\[
	J^* \ge G^\dagger.
\]

Combining both inequalities yields
\[
	J^* = G^\dagger,
\]
which completes the proof.

\end{proof}
	

	


\label{sec:optimal values are equal}

\subsection{Global Optimization Method for Problems $P_1$ and $P_2$ and Proof of their Optimality}

This subsection presents a method for obtaining the globally optimal values of problems $P_1$ and $P_2$, centered on the dynamic programming algorithm and Bellman's principle of optimality. The proof in this section does not directly address the two problems $P_0$ and $P_1$. Instead, to facilitate the derivation in the following subsection, a new problem $P_3$ is introduced, and the global optimality of the solution algorithm for problem $P_3$ is proven using Bellman's principle of optimality. 

The core logic involves introducing a random variable to characterize the approximation error of ${{{\tilde g}_k}({{\symbf{s}}_k},{{\symbf{d}}_k})}$ in problem $P_2$ relative to the original cost ${{{g}_k}({{\symbf{s}}_k},{{\symbf{d}}_k})}$. The classical theory of dynamic programming for stochastic control processes is used to establish a unified proof, which applies to both problems \( P_1 \) and \( P_2 \). At the end of this subsection, some properties of the newly introduced problem are discussed.

Note that Eq.~\eqref{eq:error} shows that ${{{\tilde g}_k}({{\symbf{s}}_k},{{\symbf{d}}_k})}$ in problem \( P_2 \) is represented as the original cost ${{{g}_k}({{\symbf{s}}_k},{{\symbf{d}}_k})}$ plus an approximation error.
Introducing a random variable $w_k$ to characterize this error, the cost function for the $k$th stage becomes ${\tilde g_k}({{\symbf{s}}_k},{{\symbf{d}}_k},{w_k}) = {g_k}({{\symbf{s}}_k},{{\symbf{d}}_k}) + \varepsilon ({w_k})$, where $\varepsilon ({w_k})$ represents the estimation error, and ${{\mathscr W}_k}({{\symbf{s}}_k},{{\symbf{d}}_k})$ represents the error space, satisfying the following relationship:

\begin{equation}
	{w_k} \in {{\mathscr W}_k}({{\symbf{s}}_k},{{\symbf{d}}_k}),\quad k = 1,2,\ldots,N
\end{equation}
The form of the error space is not specified here but will be explained at the end of this subsection.

The state transition equation becomes ${{\symbf{s}}_{k + 1}} = {{\symbf{d}}_k} + {\symbf{\theta }}_k({w_k})$, where ${\symbf{\theta}}_k({w_k})$ represents the state error.

The $N$-stage decision problem $P_3$ with random variables is expressed as:
\begin{align}
	\min J                               & = \mathop{\max}\limits_{\{ w_1,w_2,\ldots,w_{N}\} } [\sum_{k = 1}^{N} \tilde{g}_k(\symbf{s}_k,\symbf{d}_k,w_k) ] \\
	\text{s.t.} \quad
	\tilde{g}_k(\symbf{s}_k,\symbf{d}_k) & = g_k(\symbf{s}_k,\symbf{d}_k) + \varepsilon (w_k) \quad k = 1,2, \ldots ,N \notag                                                              
	\\ \symbf{s}_{k + 1} &= \symbf{d}_k + \symbf{\theta}_k(w_k) \notag 
\end{align}

\begin{proposition}
	\label{prof_DP}
	For problem $P_3$, i.e., $\min J = \mathop {\max }\limits_{\{ {w_1},{w_2},\ldots,{w_{N}}\} } [\sum\limits_{k = 1}^{N } {{{\tilde g}_k}({{\symbf{s}}_k},{{\symbf{d}}_k},{w_k})} ]$, if the optimal solution exists, for $\forall {{\symbf{s}}_1} \in {{\mathscr S}_1}$, the following algorithm is obtained: for $k = 1,2, \ldots ,N$,
	\begin{equation}
		\label{eq:DP_proof}
		J_k^{[\mathrm{DP}]}\left( \symbf{s}_k \right) = \min_{\symbf{d}_k \in \mathscr{D}_k} \max_{w_k \in \mathscr{W}_k(\symbf{s}_k,\symbf{d}_k)} \left[ \tilde{g}_k(\symbf{s}_k,\symbf{d}_k,w_k) + J_{k + 1}^{[\mathrm{DP}]}( \symbf{d}_k + \symbf{\theta}_k(w_k)) \right].
	\end{equation}{}
	Specifically, $J_{N+1}^{[\mathrm{DP}]}\left( \symbf{s}_{N+1} \right) = 0$.

	If $J_k^{[{\rm{DP}}]}\left( {{{\symbf{s}}_k}} \right)$ is bounded, then the stage values obtained from Eq. \eqref{eq:DP_proof} are the globally optimal values of $P_3$, i.e.,
	\[{J^*}({{\symbf{s}}_1}) = \mathop {\min }\limits_{\left\{ {{{\symbf{d}}_1},{{\symbf{d}}_2}, \ldots ,{{\symbf{d}}_{N}}} \right\}} \mathop {\max }\limits_{\{ {w_1},{w_2},...,{w_{N}}\} } \left[ { \sum\limits_{k = 1}^{N} {{{\tilde g}_k}({{\symbf{s}}_k},{{\symbf{d}}_k},{w_k})} } \right] = J_{_1}^{[{\rm{DP}}]}\left( {{{\symbf{s}}_1}} \right)\].
\end{proposition}

\begin{proof}
	First, define $J_k^*({{\symbf{s}}_k})$ as the globally optimal value of the decision problem from stage $k$ to the end of stage $N$, which spans ($N+1-k$) stages. For $\forall {{\symbf{s}}_k} \in {{\mathscr S}_k}$, $k = 1,2, \cdots ,N$, we have
	\begin{equation}
		J_k^*\left( {{{\symbf{s}}_k}} \right) = \mathop {\min }\limits_{\left\{ {{{\symbf{d}}_k},{{\symbf{d}}_{k + 1}}, \cdots ,{{\symbf{d}}_{N}}} \right\}} \mathop {\max }\limits_{\{ {w_k},{w_{k + 1}},...,{w_{N}}\} } \left[ { \sum\limits_{i = k}^{N} {{{\tilde g}_i}({{\symbf{s}}_i},{{\symbf{d}}_i},{w_i})} } \right].
	\end{equation}
	Specifically, $J_{N+1}^{*}\left( \symbf{s}_{N+1} \right) = 0$.

	By definition, ${J^*}({{\symbf{s}}_1}) = J_1^*({{\symbf{s}}_1})$. If each subproblem satisfies $J_k^*({{\symbf{s}}_k}) = J_k^{[{\rm{DP}}]}({{\symbf{s}}_k})$, $k = 1,2, \cdots ,N+1 $, then the proposition holds.

	Therefore, we use mathematical induction to prove $J_k^*({{\symbf{s}}_k}) = J_k^{[{\rm{DP}}]}({{\symbf{s}}_k})$, $k = 1,2, \cdots ,N+1$, for $\forall {{\symbf{s}}_k} \in {{\mathscr S}_k}$.

	When $i=N+1$, by definition, $J_{N+1}^*({{\symbf{s}}_{N+1}}) = J_{N+1}^{[{\rm{DP}}]}({{\symbf{s}}_{N+1}}) = 0$ holds. Assume that when $i=k+1$, $J_{k + 1}^*({{\symbf{s}}_{k + 1}}) = J_{k + 1}^{[{\rm{DP}}]}({{\symbf{s}}_{k + 1}})$ holds for $\forall {{\symbf{s}}_{k + 1}} \in {{\mathscr S}_{k + 1}}$.

	Then, for $i=k$ and for $\forall {{\symbf{s}}_k} \in {{\mathscr S}_k}$, we have:
	\begin{align}
		J_k^*\left( \symbf{s}_k \right)
		 & = \min_{\{ \symbf{d}_k, \symbf{d}_{k + 1}, \ldots, \symbf{d}_{N } \}} \max_{\{ w_k, w_{k + 1}, \ldots, w_{N } \}} \bigg[  \sum_{i = k}^{N } \tilde{g}_i(\symbf{s}_i, \symbf{d}_i, w_i) \bigg]                                \\
		 & = \min_{\{ \symbf{d}_k, \symbf{d}_{k + 1}, \ldots, \symbf{d}_{N} \}} \max_{\{ w_k, w_{k + 1}, \ldots, w_{N } \}} \bigg[ \tilde{g}_k(\symbf{s}_k, \symbf{d}_k, w_k) + \sum_{i = k + 1}^{N} \tilde{g}_i(\symbf{s}_i, \symbf{d}_i, w_i) \bigg]                                                                                                                       \\
		 & = \min_{\{ \symbf{d}_k, \symbf{d}_{k + 1}, \ldots, \symbf{d}_{N} \}} \max_{w_k \in \mathscr{W}_k(\symbf{s}_k, \symbf{d}_k)} \Bigg\{  \tilde{g}_k(\symbf{s}_k, \symbf{d}_k, w_k) +  \max_{\{ w_{k + 1}, \ldots, w_{N } \}} \bigg[ \sum_{i = k + 1}^{N} \tilde{g}_i(\symbf{s}_i, \symbf{d}_i, w_i) \bigg]  \Bigg\}                                                  \\
		& = \min_{\symbf{d}_k \in \mathscr{D}_k} \Bigg\{ \min_{\{ \symbf{d}_{k + 1}, \ldots, \symbf{d}_{N} \}} \max_{w_k \in \mathscr{W}_k(\symbf{s}_k, \symbf{d}_k(\symbf{s}_k))}
		\bigg[ \tilde{g}_k(\symbf{s}_k, \symbf{d}_k(\symbf{s}_k), w_k) + \max_{\{ w_{k + 1}, \ldots, w_{N} \}}  \sum_{i = k + 1}^{N}  \tilde{g}_i(\symbf{s}_i, \symbf{\mu}_i(\symbf{s}_i), w_i) \bigg] \Bigg\} \label{eq:dp_last_line_not_prooved}
	\end{align}

	Next, using Lemma \ref{lemma:1}, we interchange the min and max operators. Let ${\symbf{w}} = {w_k}$, 
	${\symbf{\mu }} = [{{\symbf{d}}_k}; {{\symbf{d}}_{k + 1}}; \ldots; {{\symbf{d}}_{N - 1}}]$, 
	${\symbf{d}} = [{{\symbf{d}}_k}; {{\symbf{d}}_{k + 1}}; \ldots; {{\symbf{d}}_{N - 1}}]$, and ${\symbf{f}}({\symbf{w}}) = {{\symbf{s}}_{k + 1}}$.
	\begin{align*}
		G_0(\symbf{w})                      & =
		\begin{cases}
			\tilde{g}_k(\symbf{s}_k,\symbf{d}_k(\symbf{s}_k),w_k), & w_k \in \mathscr{W}_k(\symbf{s}_k,\symbf{d}_k) \\
			\infty,                                                  & \text{otherwise}
		\end{cases} \\
		G_1(\symbf{f}(\symbf{w}),\symbf{d}) & =
		\begin{cases}
			\max\limits_{\{ w_{k + 1}, \ldots, w_{N } \}}  \sum\limits_{i = k + 1}^{N}  \tilde{g}_i(\symbf{s}_i, \symbf{d}_i(\symbf{s}_i), w_i) ,
			 & w_i \in \mathscr{W}_i(\symbf{s}_i,\symbf{d}_i) \\
			\infty,
			 & \text{otherwise}
		\end{cases}
	\end{align*}

	According to Lemma \ref{lemma:1}, Eq. \eqref{eq:dp_last_line_not_prooved} is further transformed into
	\begin{align}
		J_k^*\left( \symbf{s}_k \right)
		 & = \min_{{\symbf{d}_k \in \mathscr{D}_k}} \max_{{w_k \in \mathscr{W}_k(\symbf{s}_k,\symbf{d}_k)}}
		\Bigg\{ \tilde{g}_k(\symbf{s}_k,\symbf{d}_k,w_k) +  \min_{\{\symbf{d}_{k+1}, \ldots, \symbf{d}_{N}\}} \max_{\{w_{k+1},\ldots,w_{N}\}}
		\bigg[  \sum_{i=k+1}^{N} \tilde{g}_i(\symbf{s}_i,\symbf{d}_i,w_i) \bigg] \Bigg\}                                                                                                                                        \\
		 & = \min_{{\symbf{d}_k \in \mathscr{D}_k}} \max_{{w_k \in \mathscr{W}_k(\symbf{s}_k,\symbf{d}_k)}} \left[ \tilde{g}_k(\symbf{s}_k,\symbf{d}_k,w_k) + J_{k+1}^*(\symbf{s}_{k+1}) \right]                                               \\
		 & = \min_{{\symbf{d}_k \in \mathscr{D}_k}} \max_{{w_k \in \mathscr{W}_k(\symbf{s}_k,\symbf{d}_k)}} \left[ \tilde{g}_k(\symbf{s}_k,\symbf{d}_k,w_k) + J_{k+1}^{\mathrm{[DP]}}(\symbf{s}_{k+1}) \right]                                 \\
		 & = \min_{{\symbf{d}_k \in \mathscr{D}_k}} \max_{{w_k \in \mathscr{W}_k(\symbf{s}_k,\symbf{d}_k)}} \left[ \tilde{g}_k(\symbf{s}_k,\symbf{d}_k,w_k) + J_{k+1}^{\mathrm{[DP]}}(\symbf{d}_k + \symbf{\theta}_k(w_k)) \right] \\
		 & = J_k^{\mathrm{[DP]}}(\symbf{s}_k)
	\end{align}

	Therefore, the induction hypothesis holds, and the original proposition is proved.
\end{proof}

As already pointed out, the above proof does not depend on the specific form of the error space, meaning the corresponding error space can be freely specified. Thus, two important properties of problem $P_3$ are obtained as follows:
\begin{enumerate}
	\item When all $w_k$ are constantly zero, i.e., $\forall k, w_k = 0$, the corresponding error terms (estimation error or state error) will also be zero. In this special case, problem $P_3$ degenerates into problem $P_1$. This indicates that problem $P_1$ can be regarded as a special case of problem $P_3$ under no perturbation or disturbance.
	\item In problem $P_2$, the objective function value \({\tilde g_k}\) at each stage corresponds to $w_k$ and \({\tilde g_k}\) in problem $P_3$. This implies that by appropriately choosing $w_k$, a solution in problem $P_3$ can match a solution in problem $P_2$. Therefore, the solution space of problem $P_2$ is a subset of the solution space of problem $P_3$.
\end{enumerate}

From this, it can be concluded that the solution spaces of problems $P_1$ and $P_2$ are subsets of the solution space of problem $P_3$. Therefore, the difference in the globally optimal solutions between problems $P_1$ and $P_2$ will naturally be constrained by the difference in the globally optimal solutions between problems $P_1$ and $P_3$. This will facilitate the derivation in the next section.

\subsection{Derivation of the Upper Bound of the Globally Optimal Solution Error between Problems $P_0$ and $P_2$}

This section addresses another key research question. After transforming the original problem $P_0$ into a multi-stage decision problem and then discretizing and approximating it into problem $P_2$, is there any mathematical relationship between the globally optimal solutions of problem $P_2$ and the original problem $P_0$? Furthermore, what is the maximum possible discrepancy between the two? This section addresses these questions.

To avoid confusion, the globally optimal solutions of each stage for problems $P_1$, $P_2$, and $P_3$ are defined as $J_k^{[{P_1}]}\left( {{{\symbf{s}}_k}} \right)$, $J_k^{[{P_2}]}\left( {{{\symbf{s}}_k}} \right)$, and $J_k^{[{P_3}]}\left( {{{\symbf{s}}_k}} \right)$, respectively, for $k = 1, 2, \cdots, N+1$. For convenience in the subsequent derivations, ${\varepsilon _{\max }} = \max \left\{ {{\varepsilon _1}\left( {{w_1}} \right), {\varepsilon _2}\left( {{w_2}} \right), \ldots, {\varepsilon _{N }}\left( {{w_{N}}} \right)} \right\}$ is defined as the maximum error in the approximation of ${\tilde g}$ across all stages. Using the relevant formulas from the previous section's proof, the following equation holds:
\begin{align}
	J_{N+1}^{[{P_3}]}\left( {{{\symbf{s}}_{N+1}}} \right) & = J_{N+1}^{[{P_1}]}\left( {{{\symbf{s}}_{N+1}}} \right) = 0                                                                                                                                                                       \\
	J_{N}^{[P_3]}(\symbf{s}_{N})          & = \min_{\symbf{d}_{N } \in \mathscr{D}_{N }} \max_{w_{N} \in \mathscr{W}_{N}} \left[ \tilde{g}_{N}(\symbf{s}_{N},\symbf{d}_{N },w_{N }) + J_{N + 1}^{[P_3]}(\symbf{s}_{N + 1}) \right] \notag \\
	& \le \min_{\symbf{d}_{N } \in \mathscr{D}_{N }} \left[ g_{N }(\symbf{s}_{N },\symbf{d}_{N }) + \varepsilon_{\max} + J_{N + 1}^{[P_1]}(\symbf{s}_{N + 1}) \right] \notag                \\
	& = J_{N }^{[P_1]}(\symbf{s}_{N }) + \varepsilon_{\max}                                                                                                                                                              \\
	J_{N - 1}^{[P_3]}(\symbf{s}_{N - 1})          & = \min_{\symbf{d}_{N - 1} \in \mathscr{D}_{N - 1}} \max_{w_{N - 1} \in \mathscr{W}_{N - 1}} \left[ \tilde{g}_{N - 1}(\symbf{s}_{N - 1},\symbf{d}_{N - 1},w_{N - 1}) + J_N^{[P_3]}(\symbf{s}_N) \right] \notag             \\
	                                              & \le \min_{\symbf{d}_{N - 1} \in \mathscr{D}_{N - 1}} \left[ g_{N - 1}(\symbf{s}_{N - 1},\symbf{d}_{N - 1}) + \varepsilon_{\max} + J_N^{[P_1]}(\symbf{s}_N)  + \varepsilon_{\max}\right] \notag                                                 \\
	                                              & = J_{N - 1}^{[P_1]}(\symbf{s}_{N - 1}) + 2\varepsilon_{\max}                                                                                                                                                               \\
	                                              & \ldots \notag                                                                                                                                                                                                             \\
	J_1^{[{P_3}]}\left( {{{\symbf{s}}_1}} \right) & \le J_1^{[{P_1}]}\left( {{{\symbf{s}}_1}} \right) + N  {\varepsilon _{{\rm{max }}}}
\end{align}

Since the solution space of problem $P_2$ is a subset of the solution space of problem $P_3$, and according to Eq. \eqref{eq:p0=p1}, we have
\begin{equation}
	\label{eq:error_bounds}
	J_1^{[{P_2}]}\left( {{{\symbf{s}}_1}} \right) \le J_1^{[{P_1}]}\left( {{{\symbf{s}}_1}} \right) + N \cdot {\varepsilon _{{\rm{max }}}} = J_1^{[{P_0}]}\left( {{{\symbf{s}}_1}} \right) + N \cdot {\varepsilon _{{\rm{max }}}},
\end{equation}

Equation~\eqref{eq:error_bounds} is referred to as the error upper bound equation in this paper. It indicates that the final obtained solution will not differ from the globally optimal solution of the original problem by more than $N$ times the maximum single-stage error. Thus, this paper provides the globally optimal solution with theoretical guarantees for the multi-flyby problem.

\section{Analysis and Discussion}
\label{sec:analysis_discussion}

The goal of this paper is to exploit the Markov property of multi-flyby problems to develop a more efficient modeling formulation, enabling dynamic programming methods to find the global optimum. This section examines the logical interconnections between the original optimal control problem ($P_0$) and its reformulated multi-stage decision problems ($P_1$, $P_2$, and $P_3$). This reformulation facilitates error quantification and ensures globally optimal solutions. In what follows, the connections between multi-flyby and rendezvous problems and their implications for both impulsive and low-thrust trajectory design are analyzed. Finally, this section explores the computational complexity of the proposed methods and introduces practical solving techniques.

\subsection{Relationship between Problems and Explanation of Globally Optimal Solutions}

\begin{figure}
	\centering
	\includegraphics[width=0.9\linewidth]{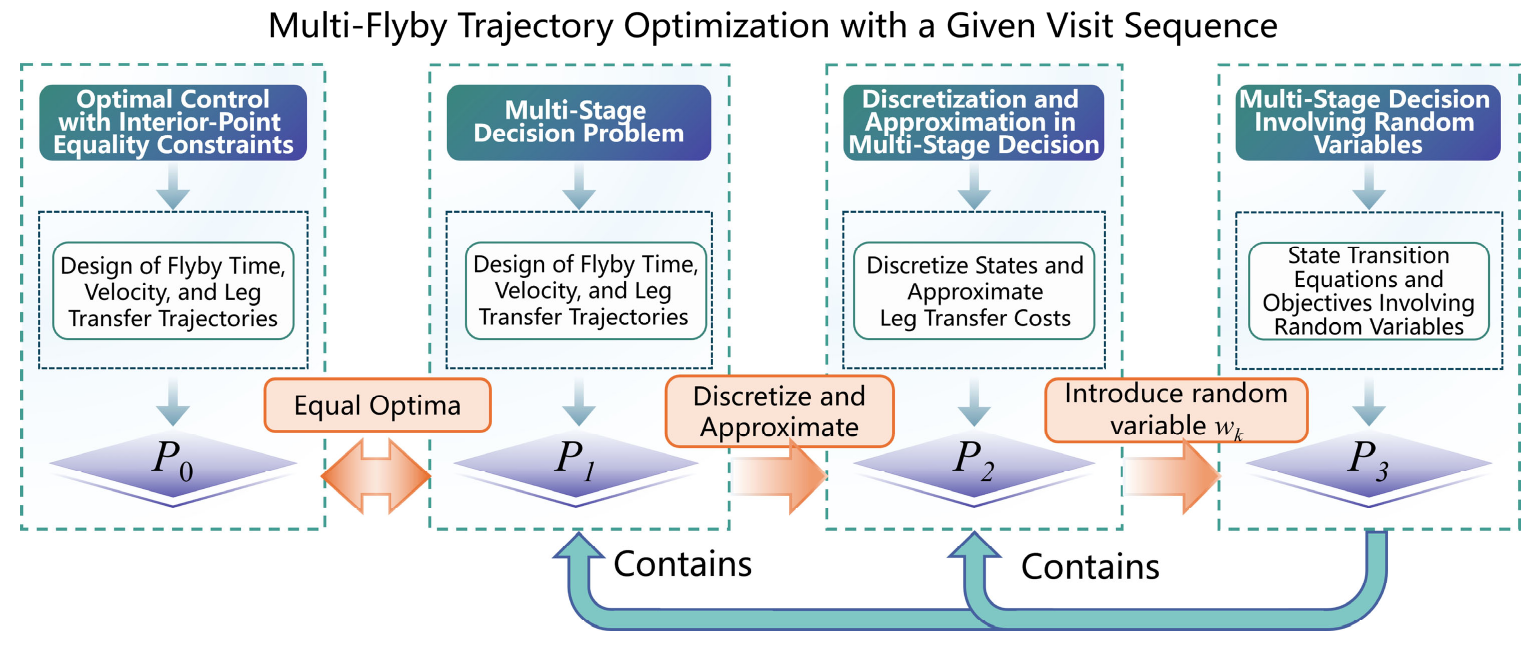}
	\caption{The logical relationships between different problems.}
	\label{fig:ch4-figre_lationship_problems}
\end{figure}

The relationships between the problems are illustrated in Fig. \ref{fig:ch4-figre_lationship_problems}. Problem $P_0$ is defined as the original optimal control problem, which involves interior point equality constraints. Problem $P_1$ is a multi-stage decision problem with the same globally optimal value as problem $P_0$. Problem $P_2$ is a discretized and approximated version of Problem $P_1$, where the transition costs and state space are approximated, making it suitable for computational implementation. Problem $P_3$ is a constructive multi-stage decision problem with random variables that aims to integrate the solution spaces of problems $P_2$ and $P_1$, as shown in Fig. \ref{fig:ch4-figre_solutionspace}. Although introducing random variables in Problem $P_3$ increases the complexity of the problem, it clarifies the relationship between the globally optimal solutions of Problems $P_1$ and $P_2$ by embedding them into a higher-dimensional solution space.

\begin{figure}
	\centering
	\includegraphics[width=0.4\linewidth]{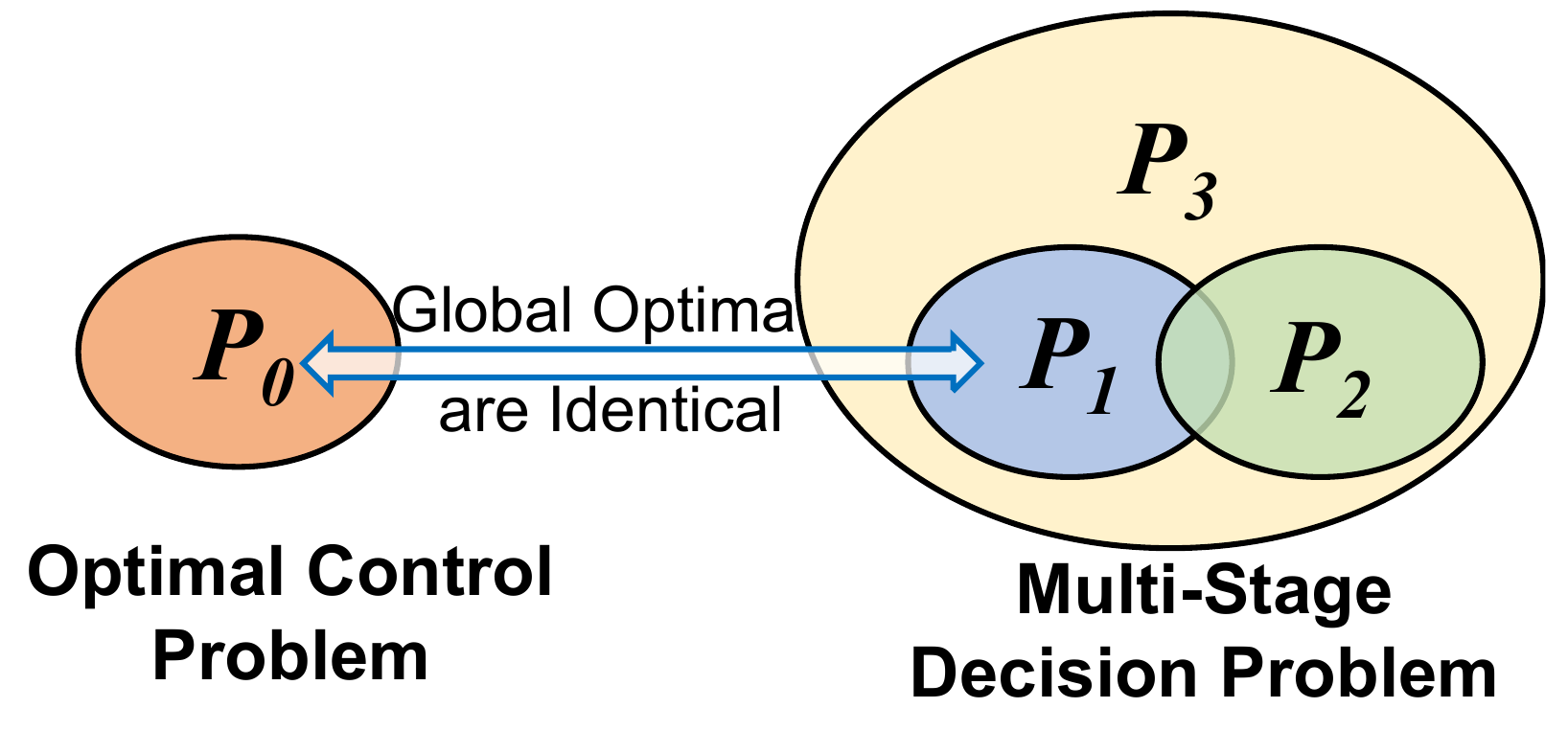}
	\caption{The solution space relationship between different problems.}
	\label{fig:ch4-figre_solutionspace}
\end{figure}

According to Proposition \ref{prof_DP}, solving problem $P_2$ using dynamic programming yields the globally optimal solution. However, the solution to problem $P_2$ obtained by this method is essentially a discretized and approximated version of the globally optimal solution to the original problem $P_1$. Therefore, Eq.~\eqref{eq:error_bounds} is introduced in this paper to quantify the error between the two solutions.

\subsection{Analysis of the Error Upper Bound}
\label{sec:error_upper_bound}

In the derivation of the error upper bound in Eq. \eqref{eq:error_bounds}, the term with $N$ times the single-stage maximum error ${\varepsilon _{\max }}$ can be simplified to the sum of the maximum errors at each stage, i.e., $\sum {\varepsilon _i}\left( {w_i} \right)$. If the optimization criterion is to minimize fuel mass consumption, the related formulas become more complex and cannot be represented as a simple linear superposition. This indicates that minimizing the total velocity increment is more suitable as the optimization criterion for dynamic programming algorithms despite the equivalence of the two problems in terms of optimization criteria.

When refining the discrete grid and improving the approximation function, i.e., as ${\varepsilon _{\max }} \to 0$, the result approaches the theoretical globally optimal solution of the original problem. Note that the upper error bound represents a theoretical limit; in practice, the worst-case scenario is rare, and approximate solutions are typically near the true global optimum. Moreover, stochastic dynamic programming theory can be applied if the random variables follow unbiased probability estimates and the error space is measurable. It can then be proven that the expected globally optimal value of problem $P_3$ equals the globally optimal value of problem $P_1$.

The significance of the error upper bound equation lies in two aspects: first, in practical mission design, if the calculated error upper bound meets the specified design target requirements, calculations can be stopped immediately, saving further work using other methods; second, easily obtainable single-state error information can be used with the error upper bound in Eq. \eqref{eq:error_bounds} to derive the error for all stages. Previously, there were few effective methods to directly evaluate the error of the entire solution, making this a key contribution of this paper.
This approach thereby provides an efficient tool for evaluating solution properties.

\subsection{Connection between Multi-Flyby and Multi-Rendezvous Problems}

The main difference in mathematical representation between the multi-flyby problem and the multi-rendezvous problem is that the rendezvous problem requires both the position and the velocity to coincide with the target at the rendezvous epoch. Hence, the multi-rendezvous problem can be regarded as a special case of the multi-flyby problem. The logical framework in this paper is an extension of the method presented in \cite{Zhang2024}, and has also been applied in designing the GTOC12 problem \cite{zhangSustainableAsteroidMining2025a}.

Unlike in the multi-flyby problem, omitting the velocity term in the multi-rendezvous problem reduces the state variable dimension by three. This reduction lowers the time complexity of the algorithm, thereby improving computational efficiency. This observation also motivates other reductions in the state variable dimension, under certain simplifications, to improve computational performance in the flyby problem. Details of these simplifications and their implications are discussed in Sec.\ref{sec:complexity_analysis}.

\subsection{Connection between Impulsive and Low-Thrust Maneuvers}
\label{sec:impulse_low_thrust}

In the multi-stage decision problem, the difference between impulsive and low-thrust propulsion is further reflected in the treatment of state variables: the mass term can be neglected in impulsive propulsion, while it must be included for low-thrust propulsion, directly affecting the spacecraft's reachability. The reachability refers to whether the spacecraft can reach the target under given initial and final positions and velocities through low-thrust propulsion.
This is feasible in the context of global optimization, as a spacecraft may consume more fuel earlier to improve its later reachability, ultimately achieving a better overall result.
Therefore, compared with impulsive propulsion, low-thrust trajectory optimization introduces additional challenges due to the need for reachability constraints.

\begin{figure}[ht]
	\centering
	\includegraphics[width=0.4\linewidth]{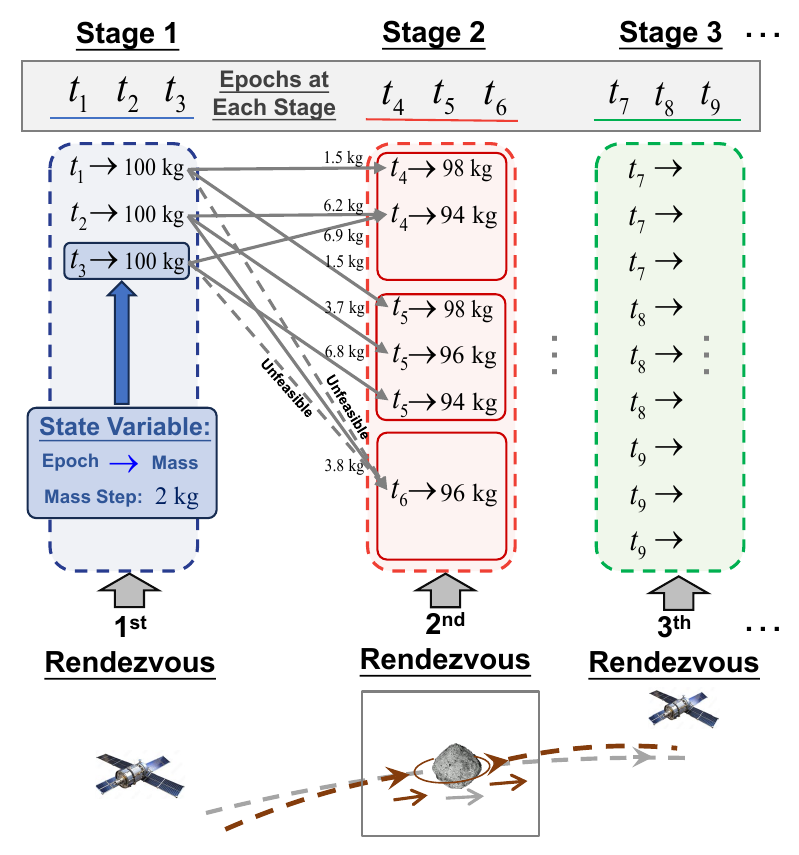}
	\caption{A schematic diagram of multi-rendezvous low-thrust dynamic programming.}
	\label{fig:rendezvous_low_throust}
\end{figure}

Including mass as a state variable may cause confusion, as the final mass in a fuel-optimal control problem is determined during the solution process rather than predefined. To clarify this process and highlight the role of mass as a state variable, a multi-rendezvous low-thrust problem is used as an example in this section, with an analogous treatment for multi-flyby problems. As shown in Fig.~\ref{fig:rendezvous_low_throust}, the spacecraft's initial mass is set to 100 kg with a step size of 2 kg. First, all possible states are generated for Stage 1. Next, for each terminal time in Stage 2 (e.g., $t_4, t_5, t_6$), the corresponding optimal control problems are solved, and the resulting fuel mass consumption (converted from $\increment v$) is recorded in the figure.
For instance, at $t_4$, the lowest fuel consumption could be 1.5 kg, while other feasible trajectories may require 6.2 kg or 6.9 kg. Hence, the state at $t_4$ is classified into two discrete mass levels: 98 kg and 94 kg. If a state in Stage 2, such as $(t_4, {\text{94 kg}})$, can be reached from multiple Stage 1 states, the predecessor yielding the lowest fuel consumption, e.g., $(t_2, {\text{100 kg}})$, is selected.

It is important to note that mass discretization is employed solely to assess the feasibility of different trajectories; no approximations are made to the actual fuel consumption. In other words, when using the state $(t_4 , \text{ 94 kg})$ as the initial condition for Stage 3, the initial mass should be the continuous value obtained from the calculation (e.g., $100 - 6.2 = 93.8$ kg), thereby minimizing the error introduced by mass discretization.

\subsection{Complexity Analysis and Practical Solving Techniques}
\label{sec:complexity_analysis}

\subsubsection{Low thrust}
\label{sec:low-thrust}

In a multi-stage decision problem consisting of $N$ stages, each with $N_s$ states, the dynamic programming algorithm generally requires updating decisions and computing state transitions for every state in each stage. This means that the computational complexity per stage is $O(N_s^{2})$. As there are $N$ such stages, the total time complexity is $O(N_s^{2}(N-1) + N_s)$.

For the multi-flyby problem, let $N_t=\lceil (T_{\rm end} -T_{\rm start})/T_{\rm step} \rceil$ represent the total number of discrete time steps in each stage, where $T_{\mathrm{step}}$ is the discrete time step size. Similarly, let $N_v$ and $N_m$ represent the total number of discrete velocities and masses in each stage, respectively. Therefore, the total number of states per stage can be expressed as $N_s = N_t N_v^3 N_m$. The computational effort of the dynamic programming process is approximately:
\[
	N_s^2 (N-1) + N_s \approx N_s^2 N = N_t^2 N_v^6 N_m^2 N.
\]

Assuming $N_t = N_v = N_m = N = 10$, this method requires $10^{11}$ calculations, highlighting the ``curse of dimensionality'' faced by dynamic programming. The following paragraphs introduce practical techniques aimed at mitigating these computational demands. Although these methods do not fully overcome the curse of dimensionality, they substantially lower the computational complexity, making the approach more feasible.

Low-thrust propulsion significantly restricts a spacecraft's reachable state region. Therefore, a natural strategy is to limit the search to states near transfer trajectories with relatively low fuel consumption. This aligns with the goal of optimizing fuel consumption and increases the likelihood of successful transfers. In extreme cases, trajectories determined solely by the dynamical environment and requiring no maneuvers are inherently reachable, regardless of thrust limitations. In the two-body problem, this corresponds to Keplerian orbits, which can be calculated using well-established Lambert solvers \cite{Russell2019,izzo2015revisiting,Russell2021}. Considering the dynamical environment with perturbations, like cislunar space \cite{topputoMeteoroidsDetectionLUMIO2023,cervoneLUMIOCubeSatObserving2022}, some improved Lambert solvers have also been widely developed for computational convenience \cite{armellin2018multiple,Yang2021,woollands2015new}.

Thus, the first step involves computing ballistic trajectories between consecutive layers and discarding those with eccentricities or inclinations apparently out of range. This step helps identify feasible arrival time intervals. Next, the velocity dimension of each layer is dynamically determined based on time. Specifically, for each epoch step in the current stage, the ballistic trajectory is computed from the average time of the previous layer to the current epoch.
The state velocities at this epoch are centered around the terminal velocity of the ballistic trajectory. Finally,  as noted in Sec.\ref{sec:impulse_low_thrust}, the terminal mass is not pre-specified but is determined based on subsequent calculations.

\subsubsection{Bi-Impulse}

\begin{figure}[ht]
	\centering
	\includegraphics[width=0.4\linewidth]{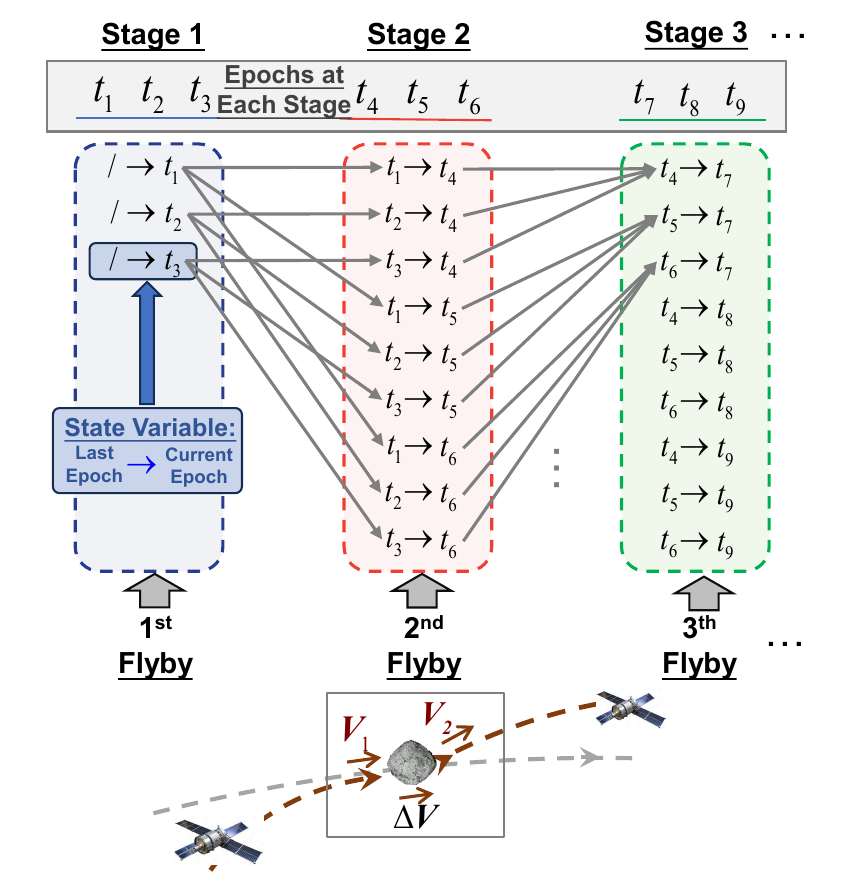}
	\caption{A schematic diagram of bi-impulse dynamic programming.}
	\label{fig:bi-impulse}
\end{figure}

This subsection introduces the bi-impulse scenario, where the dimensionality of the state variables can be further reduced. The bi-impulse scenario is defined by applying impulses only at the start and end points of the transfer trajectory without intermediate maneuvers. This scenario is widely considered in the preliminary design of deep-space missions \cite{Grigoriev2013,RUSSELL2023863}.
Specifically, a Lambert problem solver computes a transfer trajectory between the start and end points that satisfies the position conditions. The required velocity increment corresponds to the difference between the two impulses. In this case, the trajectory velocity is determined entirely by the Lambert problem solver, enabling the velocity to be excluded from the state variables and time alone to be retained.
Thus, the state variable dimensionality can be reduced to include only the current epoch and the last flyby epoch.
As a result, the state modeling of the problem simplifies into a two-dimensional state variable form, as illustrated in Fig.~\ref{fig:bi-impulse}. Furthermore, this simplification introduces time-matching characteristics that allow the computational complexity to be further reduced by leveraging these properties.

Following a derivation similar to Sec.~\ref{sec:low-thrust}, the number of discrete time points in each stage is $N_t$. In this case, the state variables, scaling as $N_t^2$ and representing the discrete time points for each stage, lead to an overall computational complexity of $O(N_t^4 N)$. As shown in Fig.~\ref{fig:bi-impulse}, the state variables include the last flyby epoch. Consequently, only the time points from the previous stage that align with the current stage's last flyby epoch are relevant.
For instance, to compute the optimal solution for the first state variable $(t_4, t_7)$ in Stage 3 of Fig.~\ref{fig:bi-impulse}, only the current time matches $t_4$ in Stage 2 are considered. State variables corresponding to other times are disregarded as they do not impact the current computation. This refinement reduces the computational complexity at this step from $N_t^2$ to $N_t$, lowering the total complexity to $O(N_t^3 N)$. To facilitate understanding of the algorithm design, a pseudocode is provided in Alg.~\ref{alg:bi-impulse}.

\begin{algorithm}
	\caption{Dynamic Programming for the Two-Impulse Multi-Flyby Problem}
	\label{alg:bi-impulse}
	\begin{algorithmic}[1]
		\REQUIRE Mission sequence
		\ENSURE Total $\Delta v$, optimal result path

		\STATE Initialize sequence length:
		$\mathbf{N} \leftarrow \text{length}(\text{sequence})$

		\STATE For the initial stage ($\mathbf{stage} = 0$), generate possible initial states based on feasible departure times:

		$\mathbf{dp\_states}[0] \leftarrow \left\{ \text{State}(t) \mid t \in \text{feasible departure times} \right\}$
		\FOR{$\mathbf{stage} = 1$ \TO $\mathbf{N} - 1$}
		\STATE Set departure body $\mathbf{Body}_{\text{departure}}$  and target body $\mathbf{Body}_{\text{target}}$
		\STATE Initialize current stage's state list: $\mathbf{dp\_states}[\mathbf{stage}] \leftarrow \emptyset$
		\STATE Determine feasible arrival times at $\mathbf{Body}_{\text{target}}$ based on time and mission constraints
		\FOR{\textbf{each} feasible arrival time $t$}
		\FOR{\textbf{each} $\mathbf{prev\_state}$ \textbf{in} $\mathbf{dp\_states}[\mathbf{stage} - 1]$}
		\STATE Compute the required $\Delta v_{\text{leg}}$ for this leg \label{alg:dp_impulse_dv}
		\STATE Compute cumulative total $\Delta v$:
		$\text{total\_dv} \leftarrow \mathbf{prev\_state}.\text{total\_dv} + \Delta v_{\text{leg}}$

		\IF{total\_dv is the smallest among those with the same prev\_state.$t$}
		\STATE Create new state $\mathbf{new\_state}$ with:
		\begin{itemize}
			\item Time $\mathbf{new\_state}.t \leftarrow t$
			\item Last time $\mathbf{new\_state}.{t_\text{last}} \leftarrow \mathbf{prev\_state}.t$
			\item Cumulative $\Delta v$ $\mathbf{new\_state}.\text{total\_dv} \leftarrow \text{total\_dv}$
			\item Predecessor state $\mathbf{new\_state}.\text{prev\_state} \leftarrow \mathbf{prev\_state} $
		\end{itemize}
		\STATE Add $\mathbf{new\_state}$ to $\mathbf{dp\_states}[\mathbf{stage}]$
		\ENDIF
		\ENDFOR
		\ENDFOR
		\ENDFOR
		\RETURN The optimal total $\Delta v$ and the optimal result path
	\end{algorithmic}
\end{algorithm}

The calculation of $\Delta {v}_\text{leg}$ in Alg.~\ref{alg:bi-impulse}, line 9, proceeds as follows: Assume the state at the previous stage is $\symbf{x}_{\text{prev}} = (t_{\text{last}}, t_{\text{departure}})$, and the state at the current stage is $\symbf{x}_{\text{current}} = (t_{\text{departure}}, t_{\text{target}})$. The celestial bodies involved at these three time points are $A_0$, $A_1$, and $A_2$, where $A_1$ and $A_2$ correspond to $\text{Body}_{\text{departure}}$ and $\text{Body}_{\text{target}}$ in Alg.~\ref{alg:bi-impulse}, respectively. The respective paths of the trajectory legs $A_0$--$A_1$ and $A_1$--$A_2$ can be determined using a Lambert problem solver. Let the velocities at $A_1$ for these two trajectory legs be $\symbf{v}_1$ (from $A_0$--$A_1$) and $\symbf{v}_2$ (from $A_1$--$A_2$). The $\Delta {v}_\text{leg}$ required is then given by:
\begin{equation}
	\label{eq:delta_v_calculation}
	\Delta {v}_\text{leg} = \| \symbf{v}_1 - \symbf{v}_2 \|.
\end{equation}

\subsubsection{Adaptive step technique}
\label{sec:adaptive_step_method}

This subsection introduces the adaptive step technique, which significantly accelerates the computation speed. Although it sacrifices the theoretical guarantee of global optimality, the subsequent case study demonstrates its practical effectiveness.

The adaptive step technique can be conceptually understood as an iterative process that performs progressively finer dynamic programming with smaller step sizes within a narrow “tube” of the previous trajectory.
The adaptive step technique begins by using a relatively large step size to obtain a rough trajectory. The step size is then reduced, and the algorithm is executed only around this trajectory.
This reduction in range ensures that, even with a smaller step size, the number of states is minimized, resulting in a fast computation speed. Iterating this process allows the step size to reach the desired threshold. However, the drawback of this method is that each iteration is conducted around the previous result. Consequently, the globally optimal solution is constrained by its localized nature, which may lead to local optima. To mitigate this risk, it is essential to maintain a relatively large area around the trajectory to achieve better results and avoid using excessively large step sizes during the initial iteration. Therefore, the two key parameters in the adaptive step technique are the initial step size and the range around the trajectory.

\section{Examples}
\label{sec:examples}

This section verifies the effectiveness of the proposed method through three scenarios: the GTOC4 impulsive problem, the GTOC11 problem, and the GTOC4 problem.
 All scenarios focus on the fuel-optimal trajectory design for multiple flybys.
The first scenario validates the impulsive problem and evaluates the performance of different step sizes and the adaptive step technique. The second scenario addresses a problem with constraints on the maximum magnitude of the flyby velocity and further analyzes the error upper bound. The third scenario focuses on low-thrust maneuvers, which involve the largest number of state variables.

The computational environment for all tests was as follows: The first two scenarios were executed on a personal laptop equipped with an AMD Ryzen 7 6800H CPU (base clock frequency ~3.2 GHz) with eight cores, 16 threads, and 16 GB of DDR5 memory. The laptop runs Windows 11, and the programs were developed in C++ using Visual Studio 2022. The code operates in serial mode without parallel computing.

For the third test scenario, cloud computing resources were utilized due to the significant computational demands. Neural networks were employed to perform low-thrust fuel consumption and reachability prediction calculations. The training procedure was conducted using Python, with PyTorch 2.0 and CUDA 11.7. Details of the training methods and model training processes are not elaborated on here as they are not central to this study. After model training, the model was imported into the C++ platform via PyTorch's C++ interface, LibTorch, and executed on a GPU to perform low-thrust fuel consumption and reachability prediction calculations. In addition to LibTorch, the authors developed all C++ software used in the third scenario. Computational resources included 2 NVIDIA RTX 3090 GPUs, each with 24 GB of VRAM, and a computer equipped with a 32-core AMD EPYC 7282 3.0 GHz CPU. The calculations were performed using OpenMP 5.0 for parallel processing. The software was executed on a Linux system and compiled with GCC 10.2.0.

\subsection{Case 1: The GTOC4 Impulsive Problem}

\begin{table}[htbp]
	\centering
	\textcolor{black}{
	\caption{Results of the GTOC4 impulsive problem with different step sizes.}
	\label{tab:step_effect_GTOC4}
	  \begin{tabular}{cccc|ccc}
	  \toprule
	  Step size, & \multicolumn{3}{c}{Fixed-step only} & \multicolumn{3}{c}{Fixed + adaptive-step refinment} \\
  \cmidrule{2-4}  \cmidrule{5-7}  days & CPU time, s & $\Delta v$, m/s & $m_{ f}$, kg & CPU time, s & $\Delta v$, m/s & $m_{ f}$, kg \\
	  \midrule
	  32 & 0.311  & 228,369.5  & 0.6  & 12.428  & 25,210.4  & 636.7  \\
	  16 & 2.212 & 120,761.0  & 24.7  & 13.817 & 25,210.4  & 636.7  \\
	  8 & 17.148  & 67,588.0  & 150.8  & 28.249  & 25,210.4  & 636.7  \\
	  4 & 133.492  & 39,620.2  & 390.1  & 145.330  & 25,210.4  & 636.7  \\
	  2 & 1,045.323  & 31,350.8  & 516.8  & 1,057.501 & 25,210.4  & 636.7  \\
	  1 & 8,396.453  & 27,002.8  & 599.1  & 8,411.756  & 25,210.4  & 636.7  \\
	  \bottomrule
	  \end{tabular}%
	  }
  \end{table}%

Case 1 focuses on the GTOC4 Impulsive Problem, which simplifies the original GTOC4 problem by using impulsive maneuvers applied only at flybys and rendezvous points, without any intermediate maneuvers. The original GTOC4 problem will be introduced in Sec.~\ref{sec:GTOC4_result}. The goal of Case 1 is to validate the effectiveness of our method for impulsive problems and to test the performance of different step sizes and the adaptive step method.

The GTOC4 impulsive problem was originally formulated by the GTOC4 champion, Moscow State University, as a simplified version during their preliminary design phase. During the GTOC4 competition, the winning team initially used bi-impulse maneuvers to approximate the global search for the flyby sequence, resulting in a preliminary trajectory design with 48 flybys and one rendezvous. Subsequently, they employed low-thrust maneuvers to refine the transfer trajectories, achieving the final winning solution with 44 flybys and one rendezvous \cite{Grigoriev2013}.

This paper further optimizes the fuel consumption based on the flyby sequence with 48 flybys obtained from the winning team's global search. To ensure a fair comparison, the same impulsive maneuver approach as used by the winning team is adopted, with changes in the mass and velocity as specified in Eq. \eqref{eq:impulse_dv}. During the calculation process, the $\Delta v$ for each stage is determined following Eq.~\eqref{eq:delta_v_calculation}. At the final stage, an additional velocity increment is required to rendezvous with the asteroid.

The influence of step sizes on the results was analyzed, and adaptive step techniques were performed. The findings are summarized in Table~\ref{tab:step_effect_GTOC4}.
The adaptive step technique involves computing the initial results with a fixed step size and then making adaptive adjustments based on these results. It functions as a supplementary process to the initial computation.

Table~\ref{tab:step_effect_GTOC4} reveals the following key observations:
\begin{itemize}
	\item First, in the fixed-step-only case, the computation time increases significantly as the step size decreases, consistent with the $O(N_t^3)$ growth predicted in Sec.~\ref{sec:adaptive_step_method}.
	\item Second, with adaptive step refinement, the final results converge consistently, demonstrating the robustness of this method. Moreover, additional computation time of the adaptive step procedure (i.e., the total time minus the fixed-step-only time) remains stable and minimal, making it suitable as a foundational algorithm for exploring other sequences.
	\item Finally, even with a step size reduced to one day, the fixed-step-only results show a notable deviation from the optimal values listed in the table. This suggests that the multi-flyby problem is highly nonlinear, characterized by numerous local optima, underscoring the need for advanced global optimization techniques. This observation is further validated in the subsequent case study.
\end{itemize}

\begin{table}[ht]
	\centering
	\caption{Detailed results of the GTOC4 impulsive problem in comparison with the winner's result\cite{Grigoriev2013}.}
	\begin{tabular}{ccccc|ccccc}
		\toprule
		    & Flyby     & Flyby time & Winner & Our      &     & Flyby     & Flyby time & Winner & Our      \\
		No. & body      & (MJD2000)  & mass, kg & mass, kg & No. & body      & (MJD2000)  & mass, kg & mass, kg \\
		\midrule
		1   & Earth     & 58,676.2   & 1,500.0  & 1,500.0  & 26  & 2005WK4   & 60,600.7   & 1,082.3  & 1,100.8  \\
		2   & 2006QV89  & 58,731.9   & 1,487.5  & 1,500.0  & 27  & 2004QJ13  & 60,665.5   & 1,073.8  & 1,095.4  \\
		3   & 2006XP4   & 58,801.2   & 1,455.5  & 1,488.7  & 28  & 2006UB17  & 60,797.0   & 1,064.1  & 1,087.1  \\
		4   & 2008EP6   & 58,867.1   & 1,447.4  & 1,454.7  & 29  & 1993FA1   & 60,898.4   & 1,047.9  & 1,078.1  \\
		5   & 2007KV2   & 58,974.1   & 1,433.9  & 1,447.1  & 30  & 2005EU2   & 60,938.4   & 1,032.0  & 1,060.2  \\
		6   & 2005XN27  & 59,084.5   & 1,415.3  & 1,434.4  & 31  & 2006EC    & 61,052.2   & 996.5    & 1,046.8  \\
		7   & 2006TB7   & 59,119.3   & 1,407.2  & 1,416.5  & 32  & 143527    & 61,082.6   & 979.8    & 1,010.2  \\
		8   & 2008AF4   & 59,222.5   & 1,369.7  & 1,406.3  & 33  & 2005GY8   & 61,109.7   & 960.4    & 994.7    \\
		9   & 2006HF6   & 59,326.9   & 1,362.4  & 1,371.8  & 34  & 199801    & 61,201.4   & 956.0    & 977.0    \\
		10  & 2003LW2   & 59,366.2   & 1,332.9  & 1,368.6  & 35  & 2005XW77  & 61,221.1   & 934.2    & 975.8    \\
		11  & 2008PK3   & 59,419.8   & 1,317.7  & 1,337.1  & 36  & 2008AP33  & 61,285.5   & 925.5    & 956.8    \\
		12  & 2007VL3   & 59,520.1   & 1,298.4  & 1,324.0  & 37  & 2004JN1   & 61,377.3   & 907.4    & 948.8    \\
		13  & 2006AN    & 59,586.1   & 1,289.1  & 1,305.0  & 38  & 2007US51  & 61,441.8   & 901.6    & 929.3    \\
		14  & 2006UQ216 & 59,685.1   & 1,287.8  & 1,296.2  & 39  & 2008QB    & 61,486.7   & 895.4    & 924.5    \\
		15  & 2006KV89  & 59,746.8   & 1,277.8  & 1,295.2  & 40  & 1995SA4   & 61,563.0   & 860.7    & 919.6    \\
		16  & 2004SA1   & 59,862.9   & 1,242.2  & 1,284.6  & 41  & 2006WR127 & 61,645.4   & 839.5    & 885.6    \\
		17  & 154276    & 59,901.2   & 1,227.9  & 1,248.8  & 42  & 2002CW11  & 61,727.3   & 826.3    & 862.3    \\
		18  & 2006AU3   & 59,955.5   & 1,219.2  & 1,235.9  & 43  & 2003WP25  & 61,807.3   & 785.4    & 849.2    \\
		19  & 2008GM2   & 60,079.5   & 1,189.1  & 1,230.3  & 44  & 2006BZ147 & 61,869.8   & 779.0    & 807.2    \\
		20  & 2008NA    & 60,156.8   & 1,172.0  & 1,197.7  & 45  & 2004RN111 & 61,922.2   & 767.9    & 799.6    \\
		21  & 2005CD69  & 60,233.3   & 1,142.6  & 1,183.6  & 46  & 2006VP13  & 61,980.4   & 763.8    & 793.8    \\
		22  & 2008KE6   & 60,334.8   & 1,126.1  & 1,152.5  & 47  & 2003AS42  & 62,055.5   & 723.4    & 785.1    \\
		23  & 2007VD184 & 60,422.9   & 1,119.3  & 1,137.3  & 48  & 2008RH1   & 62,105.6   & 705.8    & 741.1    \\
		24  & 2008EQ    & 60,507.2   & 1,101.9  & 1,130.4  & 49  & 2007HW4   & 62,151.6   & 703.2    & 729.1    \\
		25  & 2003YP3   & 60,550.1   & 1,093.4  & 1,121.0  & 50  & 2000SZ162 & 62,262.4   & 616.5    & 636.7    \\
		\bottomrule
	\end{tabular}%
	\label{tab:GTOC4_impulse_result}%
\end{table}%

In conclusion, the proposed method achieves fuel savings of 20.2 kg over the winning result, demonstrating its superior efficiency. Detailed results are provided in Table~\ref{tab:GTOC4_impulse_result}.

\subsection{Case 2: The GTOC11 Problem}

Case 2 validates that the proposed method can be extended to another impulsive GTOC problem, demonstrating its generality and robustness. It also illustrates how constraints on the flyby velocity can be handled and further tests the error upper bound.

The GTOC11 problem was proposed by the National University of Defense Technology and the Xi’an Satellite Control Center. One key objective of the GTOC11 problem is to design trajectories for ten motherships to fly by as many targets as possible while minimizing the mothership $\Delta v$. The thrust is modeled as impulsive, satisfying Eq.~(\ref{eq:impulse_dv}). The competition was won by Tsinghua University~\cite{ZHANG2023819}, and their solution will be used as the test case in Case 2.

The basic formulation of the GTOC11 problem is similar to that introduced in Sec.~\ref{sec:modeling}. However, there are additional constraints: the departure velocity relative to the Earth must not exceed 6 km/s, and the flyby velocity, i.e., the relative velocity between the mothership and the asteroid at the flyby epoch, must remain below 2 km/s. As the GTOC11 problem involves two phases and the design of the mothership trajectories constitutes the first phase, the flyby times for each asteroid must also comply with the constraints of the second phase. To decouple the problem, specific parameters from the second phase of the champion’s solution are used to compute a fixed time window. As a result, the modified results can replace the champion's mothership trajectories without violating any constraints of the GTOC11 problem.

The solving procedure for Case 2 is like that for Case 1, except that the flyby velocity is constrained. This prevents
direct computation of $\Delta v_{\text{leg}}$ using Eq.~\eqref{eq:delta_v_calculation}. Let $\symbf{v}_\text{flyby}$ represent the velocity at the flyby asteroid, while $\symbf{v}_1$ and $\symbf{v}_2$ denote the velocities before and after the flyby, respectively. The $\Delta v_\text{leg}$ for the GTOC11 problem can be calculated as
\begin{equation}
	\label{eq:delta_v_calculation_GTOC11}
	\Delta v = \min( \| \symbf{v}_1 - \symbf{v}_\text{flyby} \| + \| \symbf{v}_\text{flyby} - \symbf{v}_2 \|),
\end{equation}
where the velocity $\symbf{v}_\text{flyby}$ relative to the flyby body must not exceed 2 km/s, as required by the GTOC11 problem. Equation \eqref{eq:delta_v_calculation_GTOC11} can be analytically solved using the conic method described in \cite{ZHANG2023819}. This approach requires only $\symbf{v}_1$, $\symbf{v}_2$, and the maximum allowable flyby velocity. The state variables thus remain unchanged.

Based on the sequence obtained from the winning results, the fuel consumption was further optimized using an adaptive step technique with an initial step of 32 days. The results are summarized in Table~\ref{tab:GTOC11_result}. To ensure a fair comparison, the winning results presented here were derived using a bi-impulse scenario, consistent with the winning team’s global optimization method \cite{Zhang2024}.

The final result submitted by the winning team surpasses that in Table~\ref{tab:GTOC11_result}, as they performed extensive optimization for each leg of the mothership trajectory. This included optimizing bi-impulse maneuvers not applied at the flyby event points or introducing a third maneuver mid-course. To address potential concerns, we implemented these optimizations on our result, demonstrating that it could surpass the winning solution even with limited computational resources.

An analysis of Table~\ref{tab:GTOC11_result} reveals that the proposed method achieved superior results across all tasks. Notably, in Tasks 3 and 5, the velocity increment optimization was particularly significant, reaching over 600 m/s and further demonstrating the generality of the method. The ability to improve upon the winning results emphasizes the importance of global optimization for multi-flyby problems, given the prevalence of local optima in flyby trajectory design.

\begin{table}[htbp]
	\centering
	\begin{threeparttable}
		\caption{Results of the GTOC11 problem in comparison with the winner's result\cite{ZHANG2023819}.}
		\label{tab:GTOC11_result}
		\begin{tabular}{cccccc}
			\toprule
			Mission               & Sequence & \multicolumn{2}{c}{$\Delta v$, m/s} & {Reduced}          & {Computational }           \\
			\cmidrule{3-4}     ID & length   & Winner result\tnote{1}              & After optimization & $\Delta v$, m/s  & time, s \\
			\midrule
			1                     & 35       & 14,325.9                            & 13,944.8           & 381.1            & 4.467   \\
			2                     & 42       & 17,199.7                            & 17,131.1           & 68.6             & 4.501   \\
			3                     & 39       & 17,330.1                            & 16,704.1           & 626.0            & 4.21    \\
			4                     & 41       & 18,508.0                            & 18,351.0           & 156.9            & 4.487   \\
			5                     & 40       & 16,680.7                            & 16,033.9           & 646.8            & 3.915   \\
			6                     & 39       & 19,824.3                            & 19,721.8           & 102.4            & 3.57    \\
			7                     & 37       & 17,606.0                            & 17,352.9           & 253.2            & 3.398   \\
			8                     & 40       & 16,342.2                            & 16,331.0           & 11.2             & 3.552   \\
			9                     & 37       & 15,158.4                            & 15,025.1           & 133.3            & 3.334   \\
			10                    & 38       & 16,168.8                            & 16,126.6           & 42.1             & 3.448   \\
			\bottomrule
		\end{tabular}%
		\begin{tablenotes}
			\footnotesize
			\item[1] These values were calculated using the bi-impulse scenario.
		\end{tablenotes}
	\end{threeparttable}
\end{table}%

Next, we discuss the error bounds of this case. Since the bi-impulse method is employed, the impulsive velocity increment is calculated accurately. Thus, the error primarily arises from grid discretization due to time discretization. To evaluate this error, we conducted the following tests:
\begin{enumerate}
	\item Assumed a time step of $t_{\text{step}}$.
	\item Calculated the baseline value: We randomly generated three time points $t_0, t_1, t_2$ for the previous and next transfers, and calculated $\Delta v_{\text{leg}}$ using Eq. $\eqref{eq:delta_v_calculation_GTOC11}$. This result is considered the baseline.
	\item Calculated the approximated value: The time points $t_0, t_1, t_2$ were rounded up or down to the nearest time step, i.e., to the nearest integer multiple of $t_{\text{step}}$. The approximated $\Delta v_{\text{leg}}$ was then recalculated using Eq. $\eqref{eq:delta_v_calculation_GTOC11}$.
	\item Evaluated the error: The error was defined as the difference between the actual and baseline values, i.e., the difference between the two $\Delta v_{\text{leg}}$ calculations.
\end{enumerate}

We conducted the above tests for different $t_{\text{step}}$ values, and the results are shown in Table \ref{tab:GTOC11_error}. The transfer sequence is taken from Mission 1 of the GTOC11 winning team. For each transfer, 10,000 simulations were performed, and the mean absolute error and the maximum absolute error of the transfer data were recorded. As observed in Table~\ref{tab:GTOC11_error}, the error decreases as the step size decreases, which aligns with intuitive expectations.

\begin{table}[htbp]
	\centering
	\caption{Error statistics for the GTOC11 problem.}
	\begin{tabular}{ccc}
		\toprule
		Step size, days & Average error, m/s & Max error, m/s \\
		\midrule
		10              & 106.13             & 3,098.38       \\
		1               & 11.46              & 710.30         \\
		0.1             & 1.16               & 79.30          \\
		0.01            & 0.12               & 8.14           \\
		0.001           & 0.01               & 0.87           \\
		\bottomrule
	\end{tabular}%
	\label{tab:GTOC11_error}%
\end{table}%

To determine the upper bound of the final error in the GTOC11 results, a step size of 0.01 days was employed, requiring more computational resources and the refined filtering of infeasible states. The final results were almost identical to those obtained using adaptive step sizes, as shown in Table~\ref{tab:GTOC11_error}. This means that the error upper bound in our results is approximately $8\times 35 \approx 300$ m/s. Therefore, the trajectory for GTOC11 Mission 1 requires at least 13,640 m/s, providing a benchmark for future research. As previously described, this method establishes a general error upper bound based on the algorithm itself, marking a novel contribution in this domain.
Furthermore, the function of Table~\ref{tab:GTOC11_error} is not only to provide a forward estimation of the final error but also to allow for the determination of the required step size based on the desired accuracy of the results.

\subsection{Case 3: GTOC4 Problem}
\label{sec:GTOC4_result}

The purpose of Case 3 is to evaluate the effectiveness of the proposed method in low-thrust trajectory optimization problems, thereby verifying its general applicability across two thrust patterns. The test scenario is based on the original GTOC4 Problem.

The GTOC4 problem, proposed by the French National Center for Space Studies (CNES)
\footnote{Bertrand, R., Epenoy, R., and Meyssignac, B., “Problem Description for the 4th Global Trajectory Optimisation Competition,” 2009, \url{https://sophia.estec.esa.int/gtoc\_portal/wp-content/uploads/2012/11/gtoc4\_problem\_description.pdf} [retrieved 1 December 2024]},
 is focused on finding an optimal asteroid flyby sequence.  A spacecraft equipped with a low-thrust propulsion system departs from Earth, performs multiple flybys of asteroids, and ultimately rendezvouses with a final target. The primary objective is to maximize the number of flybys, and the secondary objective is to maximize the final spacecraft mass. When the spacecraft departs from Earth at the departure epoch \(t_s\), it satisfies bounded residual hyperbolic velocity constraints:
$
	\symbf{r}(t_s) - \symbf{R}_{\text{E}}(t_s) = \symbf{0}$, $\increment V_{\text{E}} < \increment V_{\text{E}}^{\text{max}}$, and $m(t_s) = m_0
$,
where \(\increment {V}_{\text{E}} = \sqrt{\increment u_{\text{E}}^2 + \increment v_{\text{E}}^2 + \increment w_{\text{E}}^2}\), with \(\increment u_{\text{E}} = u(t_s) - U_{\text{E}}(t_s)\), \(\increment v_{\text{E}} = v(t_s) - V_{\text{E}}(t_s)\), and \(\increment w_{\text{E}} = w(t_s) - W_{\text{E}}(t_s)\). The quantities \(\symbf{R}_{\text{E}}(t)\) and \(\symbf{V}_{\text{E}}(t)\) represent the position and velocity vectors of Earth at time \(t\), respectively.
At the end of the mission epoch \(t_f\), the spacecraft is required to rendezvous with a final asteroid with
$	\symbf{r}(t_f) = \symbf{R}_N(t_f)$ and
$	\symbf{v}(t_f) = \symbf{V}_N(t_f)$.
The total flight time and the final mass of the spacecraft are constrained as follows: $t_f - t_s \leq t_{\text{max}}$ and $m(t_f) \geq m_{\text{min}}$. The relevant parameters of the GTOC4 problem are summarized in Table~\ref{tab:GTOC4_parameters}.

\begin{table}[h]
	\centering
	\caption{Summary of parameters for the GTOC4 problem.}
	\label{tab:GTOC4_parameters}
	\begin{tabular}{ccc}
		\toprule
		Name                                     & Unit                      & Value                            \\
		\hline
		\(\mu\)                                  & \(\text{m}^3/\text{s}^2\) & \(1.32712440018 \times 10^{20}\) \\
		\(I_{\text{sp}}\)                        & \(\text{s}\)              & \(3000\)                         \\
		\(g_0\)                                  & \(\text{m/s}^2\)          & \(9.80665\)                      \\
		\(T_{\text{max}}\)                       & \(\text{N}\)              & \(0.135\)                        \\
		\(\increment V_{\text{E}}^{\text{max}}\) & \(\text{m/s}\)            & \(4000\)                         \\
		\(m_0\)                                  & \(\text{kg}\)             & \(1500\)                         \\
		\(t_{\text{max}}\)                       & \(\text{years}\)          & \(10\)                           \\
		\(m_{\text{min}}\)                       & \(\text{kg}\)             & \(500\)                          \\
		\bottomrule
	\end{tabular}
\end{table}

\begin{figure}[htb]
	\centering
	\includegraphics[width=0.7\linewidth]{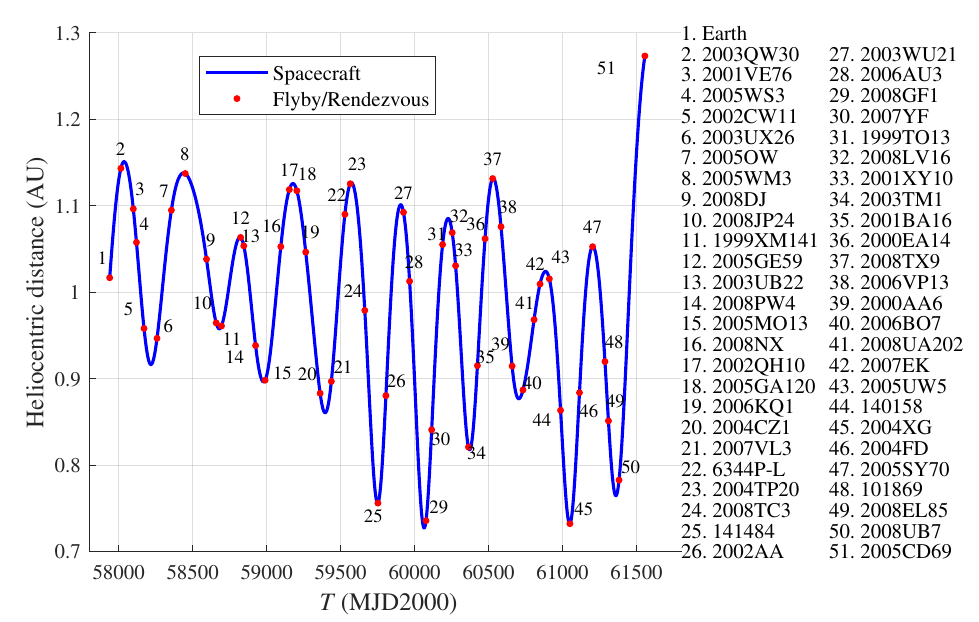}
	\caption{Heliocentric distance over time of the optimal solution for the GTOC4 problem.}
	\label{fig:ch4-figre_sundistance}
\end{figure}

\begin{figure}[htb]
	\centering
	\includegraphics[width=0.4\linewidth]{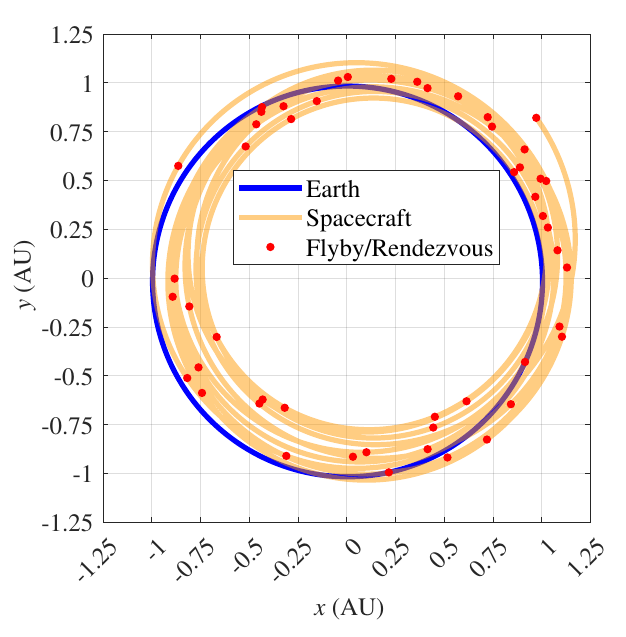}
	\caption{ The Optimal trajectory for the GTOC4 problem}
	\label{fig:ch4-figre_GTOC4_trajectory}
\end{figure}

In the GTOC4 ranking, the Moscow State University team achieved first place by performing 44 flybys and one rendezvous. Subsequently, the University of Jena team proposed an improved solution with 49 flybys and one rendezvous, representing the best-known result to date.\footnote{\url{https://www.youtube.com/watch?v=7QxikroB-6Q}}
This case study further optimizes fuel consumption based on the optimal sequence discovered by the University of Jena. During the proposed method simulation, the velocity step size was 50 m/s, the time step size was 0.2 days, the mass step size was 1 kg, and the adaptive step technique was enabled. During the calculation of the required velocity increments for transfers, a neural-network-based estimator and reachability prediction method were employed. As this aspect is unrelated to the main focus of this paper, it will be detailed in a separate article.

The total runtime of the simulation was 35.9 hours. After obtaining approximate global optimal flyby times and flyby velocities
using the velocity increment estimation and reachability prediction, an indirect method was employed to calculate the true fuel optimal transfer trajectories. The final optimization results showed a remaining 19.9 kg of fuel, surpassing the previously published results from the University of Jena and representing the currently known optimal solution. The variation of heliocentric distance over time and the two-dimensional planar trajectory of the final result are shown in Figs. \ref{fig:ch4-figre_sundistance} and \ref{fig:ch4-figre_GTOC4_trajectory}, respectively.

It should be noted that low-thrust reachability prediction inevitably has some errors. When dynamic programming selects an optimal path based on prediction results, it may be found that the result is not reachable when actually solving the optimal control problem. This implies a risk of actual non-reachability in low-thrust reachability prediction. Therefore, this study uses neural network prediction methods, and the results reflect the reachability probability. By artificially setting a high probability threshold, 
the results can ensure a certain level of robustness, thus reducing the risk of infeasibility in the actual optimal control problem.

\begin{figure}[ht]
	\centering
	\includegraphics[width=0.55\linewidth]{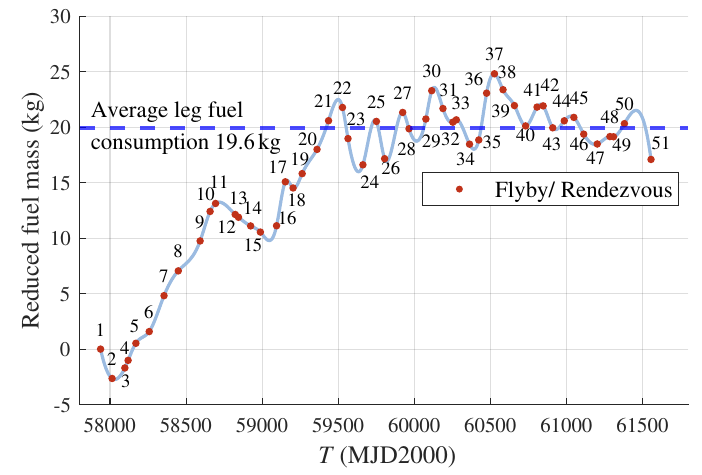}
	\caption{Fuel consumption improvement over the previous GTOC4 solution.}
	\label{fig:ch4-figure_GTOC4massrefined}
\end{figure}
Figure \ref{fig:ch4-figure_GTOC4massrefined} compares our result with the previous optimal solution. It is observed that the maximum fuel consumption is reduced by 25 kg compared to the original solution, surpassing the average leg fuel consumption of 19.6 kg. This suggests the possibility of solutions with additional flybys.

Furthermore, Fig.~\ref{fig:ch4-figure_GTOC4massrefined} shows that our result consumes more initial fuel than the previous solution. This gap often prevents local optimizers from identifying better solutions. This underscores the necessity of global optimization for such problems again and highlights the superiority of the proposed method. Detailed data from our final result can be found in Table \ref{tab:longtable}.

The error upper bound includes two primary error components introduced by the velocity increment estimator: grid discretization and estimator-induced errors. The testing procedure mirrors that of Case 2, where the true value is computed using the indirect method in the second step, and the approximate value is derived using a predictor in the third step. As a result, the estimated error upper bound for this problem is approximately $81 \times 51 \approx 4000$ m/s. This relatively high value is mainly due to the high dimensionality of state variables in low-thrust trajectory computations, which constrains further reduction of the discretization step size. As discussed in Sec.\ref{sec:error_upper_bound}, this error upper bound is conservative, and the actual error is likely smaller. The notable reduction in fuel consumption compared to the previous best-known result supports this argument.

\begin{table}[ht]
	\centering
	\footnotesize
	\caption{Optimized results for the GTOC4 problem.}
	\label{tab:longtable}
	\begin{tabular}{ccccc|ccccc}
		\toprule
		    & Flyby     & Flyby time & Relative flyby          & Mass,   &     & Flyby     & Flyby time & Relative flyby           & Mass, \\
		No. & body      & (MJD2000)  & velocity, km/s          & kg      & No. & body      & (MJD2000)  & velocity, km/s           & kg    \\
		\midrule
		1   & Earth     & 57,938.2   & (0.694,-3.939,-0.025)   & 1,500.0 & 27  & 2003WU21  & 59,923.4   & (0.690,13.216,-11.099)   & 964.7 \\
		2   & 2003QW30  & 58,013.9   & (9.753,-1.418,4.826)    & 1,488.9 & 28  & 2006AU3   & 59,964.8   & (9.007,0.575,-1.634)     & 962.3 \\
		3   & 2001VE76  & 58,097.7   & (9.168,5.250,1.649)     & 1,462.7 & 29  & 2008GF1   & 60,075.7   & (-7.338,-10.823,1.467)   & 945.0 \\
		4   & 2005WS3   & 58,119.1   & (-11.193,1.532,-7.589)  & 1,456.8 & 30  & 2007YF    & 60,114.1   & (2.574,-6.319,0.009)     & 936.5 \\
		5   & 2002CW11  & 58,170.2   & (1.011,-9.159,-1.134)   & 1,439.5 & 31  & 1999TO13  & 60,187.9   & (-5.678,-5.190,-11.777)  & 908.4 \\
		6   & 2003UX26  & 58,257.8   & (-10.095,-8.307,2.846)  & 1,411.0 & 32  & 2008LV16  & 60,252.9   & (-1.489,-18.242,1.535)   & 883.0 \\
		7   & 2005OW    & 58,355.1   & (-1.076,-10.647,-0.439) & 1,385.9 & 33  & 2001XY10  & 60,275.8   & (-6.362,-13.365,-13.514) & 874.7 \\
		8   & 2005WM3   & 58,448.7   & (11.690,-1.369,0.267)   & 1,358.5 & 34  & 2003TM1   & 60,362.5   & (-15.269,-3.462,0.948)   & 840.8 \\
		9   & 2008DJ    & 58,592.3   & (15.920,-6.879,3.126)   & 1,309.7 & 35  & 2001BA16  & 60,423.2   & (0.828,-9.252,4.449)     & 822.4 \\
		10  & 2008JP24  & 58,657.8   & (4.638,4.967,1.190)     & 1,290.1 & 36  & 2000EA14  & 60,474.7   & (6.728,-7.389,1.176)     & 821.0 \\
		11  & 1999XM141 & 58,693.8   & (4.751,8.492,11.993)    & 1,277.2 & 37  & 2008TX9   & 60,526.3   & (11.577,-5.175,3.671)    & 804.5 \\
		12  & 2005GE59  & 58,822.9   & (14.096,-0.371,8.635)   & 1,227.6 & 38  & 2006VP13  & 60,582.6   & (1.555,-1.419,4.847)     & 782.2 \\
		13  & 2003UB22  & 58,844.4   & (-4.988,-5.108,7.899)   & 1,219.2 & 39  & 2000AA6   & 60,657.6   & (-7.278,10.842,0.660)    & 752.5 \\
		14  & 2008PW4   & 58,924.1   & (4.491,-8.064,-1.197)   & 1,188.3 & 40  & 2006BO7   & 60,731.1   & (1.251,5.736,0.159)      & 724.0 \\
		15  & 2005MO13  & 58,988.3   & (12.426,-3.602,-2.626)  & 1,166.3 & 41  & 2008UA202 & 60,806.0   & (0.241,-3.227,0.280)     & 707.6 \\
		16  & 2008NX    & 59,094.6   & (-2.306,-2.948,-1.110)  & 1,147.0 & 42  & 2007EK    & 60,846.3   & (-6.545,1.557,0.835)     & 700.8 \\
		17  & 2002QH10  & 59,152.2   & (7.030,-8.765,-3.757)   & 1,139.8 & 43  & 2005UW5   & 60,909.1   & (9.420,1.314,-2.057)     & 675.9 \\
		18  & 2005GA120 & 59,203.0   & (9.192,4.187,5.149)     & 1,137.5 & 44  & 140158    & 60,985.1   & (1.394,7.182,-1.639)     & 655.0 \\
		19  & 2006KQ1   & 59,262.9   & (2.219,-1.046,-4.639)   & 1,126.1 & 45  & 2004XG    & 61,048.1   & (-8.459,-6.237,-0.453)   & 632.0 \\
		20  & 2004CZ1   & 59,360.4   & (-4.440,4.348,-0.061)   & 1,092.9 & 46  & 2004FD    & 61,112.9   & (16.830,-23.171,-0.011)  & 616.2 \\
		21  & 2007VL3   & 59,436.1   & (-1.534,-12.798,-0.957) & 1,071.8 & 47  & 2005SY70  & 61,201.8   & (2.694,-9.228,-0.853)    & 606.1 \\
		22  & 6344P-L   & 59,527.6   & (15.789,-5.121,1.983)   & 1,050.8 & 48  & 101869    & 61,285.0   & (5.888,-22.068,-2.794)   & 591.7 \\
		23  & 2004TP20  & 59,564.0   & (8.454,5.919,11.466)    & 1,043.1 & 49  & 2008EL85  & 61,308.6   & (11.045,-2.645,-1.286)   & 590.1 \\
		24  & 2008TC3   & 59,662.0   & (8.640,-8.505,0.057)    & 1,009.5 & 50  & 2008UB7   & 61,380.3   & (11.313,12.717,-0.591)   & 581.7 \\
		25  & 141484    & 59,750.7   & (-3.490,-13.572,-8.079) & 990.8   & 51  & 2005CD69  & 61,555.0   & (0.000,0.000,-0.000)     & 519.9 \\
		26  & 2002AA    & 59,804.8   & (7.977,-10.817,-5.912)  & 986.6   &     &           &            &                          &       \\
		\bottomrule
	\end{tabular}%
\end{table}%

\section{Conclusion}
\label{Sec:conclusion}

This paper has presented a method for achieving global optimality in the design of multi-flyby asteroid trajectories under a given sequence. By reformulating the original optimal control problem with intermediate equality constraints into a multi-stage decision problem, Bellman’s principle was employed to derive and prove global optimality. 
This reformulation simplifies the handling of complex trajectory constraints while providing provable error bounds, ensuring confidence in the approximate global optimum for flyby epochs and velocities. 
 The method accommodates both impulsive and low-thrust propulsion models and addresses mission constraints, such as limiting the magnitude of relative flyby velocities.
Empirical results underscore the versatility and effectiveness of the proposed method. Applied to three benchmark GTOC problems, it improves upon known best solutions, validating the approach and demonstrating its potential for trajectory design.

In summary, this work offers a computationally tractable method for global trajectory optimization in multi-flyby missions.
The foundation of this method lies in the exponential growth of computational capabilities described by Moore's Law. Tasks that were previously computationally infeasible have now become achievable, and future advancements in computing power will further highlight the advantages of this approach.
Looking ahead, this method can serve as a useful tool for mission designers seeking global optima in multi-asteroid exploration.

\section*{Appendix: Lemmas for the Proof of Proposition 1}
\label{Sec:appendix}
A mathematical lemma from the theory of dynamic programming for stochastic optimal control processes \cite{bertsekasDynamicProgrammingOptimal2012} is introduced to facilitate the proofs of Proposition 1.

\begin{lemma}
	\label{lemma:1}
	Let \(\symbf{f}:\mathscr{W} \to \mathscr{S}\) be a function, where \(\mathscr{W}\) is its domain and \(\mathscr{S}\) is its range. Let \(\mathscr{M}\) be the set of functions \(\symbf{\mu}:\mathscr{S} \to \mathscr{D}\), where \(\mathscr{W}\), \(\mathscr{S}\), and \(\mathscr{D}\) are different sets. Given any function \(G_0:\mathscr{W} \to (-\infty, \infty]\) and \(G_1:\mathscr{S} \times \mathscr{D} \to (-\infty, \infty]\), if for all \(\symbf{d} \in \mathscr{D}\), there exists \(\min\limits_{\symbf{d} \in \mathscr{D}} G_1(\symbf{f}(\symbf{w}), \symbf{d}) > -\infty\), then
	\[
		\min_{\symbf{\mu} \in \mathscr{M}} \max_{\symbf{w} \in \mathscr{W}} \left[ G_0(\symbf{w}) + G_1(\symbf{f}(\symbf{w}), \symbf{\mu}(\symbf{f}(\symbf{w}))) \right] = \max_{\symbf{w} \in \mathscr{W}} \left[ G_0(\symbf{w}) + \min_{\symbf{d} \in \mathscr{D}} G_1(\symbf{f}(\symbf{w}), \symbf{d}) \right].
	\]
\end{lemma}

\begin{proof}
	Consider any \(\symbf{\mu} \in \mathscr{M}\). For all \(\symbf{w} \in \mathscr{W}\), the following inequality holds:
	\[
		\max_{\symbf{w} \in \mathscr{W}} \left[ G_0(\symbf{w}) + G_1(\symbf{f}(\symbf{w}), \symbf{\mu}(\symbf{f}(\symbf{w}))) \right] \ge \max_{\symbf{w} \in \mathscr{W}} \left[ G_0(\symbf{w}) + \min_{\symbf{d} \in \mathscr{D}} G_1(\symbf{f}(\symbf{w}), \symbf{d}) \right].
	\]
	Thus, taking the minimum on the left-hand side of the above inequality still holds, that is,
	\[
		\min_{\symbf{\mu} \in \mathscr{M}} \max_{\symbf{w} \in \mathscr{W}} \left[ G_0(\symbf{w}) + G_1(\symbf{f}(\symbf{w}), \symbf{\mu}(\symbf{f}(\symbf{w}))) \right] \ge \max_{\symbf{w} \in \mathscr{W}} \left[ G_0(\symbf{w}) + \min_{\symbf{d} \in \mathscr{D}} G_1(\symbf{f}(\symbf{w}), \symbf{d}) \right].
	\]

	Next, we prove the inequality in the other direction, ``left-hand side \(\leq\) right-hand side'':

	For any positive number \(\varepsilon > 0\), there exists \({\symbf{\mu}_\varepsilon} \in \mathscr{M}\) such that for all \(\symbf{w} \in \mathscr{W}\), the following inequality holds:
	\[
		G_1\left( \symbf{f}(\symbf{w}), {\symbf{\mu}_\varepsilon}(\symbf{f}(\symbf{w})) \right) \le \min_{\symbf{d} \in \mathscr{D}} G_1(\symbf{f}(\symbf{w}), \symbf{d}) + \varepsilon.
	\]
	Therefore,
	\[
		\begin{aligned}
			 & \min_{\symbf{\mu} \in \mathscr{M}} \max_{\symbf{w} \in \mathscr{W}} \left[ G_0(\symbf{w}) + G_1(\symbf{f}(\symbf{w}), \symbf{\mu}(\symbf{f}(\symbf{w}))) \right]   \\
			 & \quad \le \max_{\symbf{w} \in \mathscr{W}} \left[ G_0(\symbf{w}) + G_1\left( \symbf{f}(\symbf{w}), {\symbf{\mu}_\varepsilon}(\symbf{f}(\symbf{w})) \right) \right] \\
			 & \quad \le \max_{\symbf{w} \in \mathscr{W}} \left[ G_0(\symbf{w}) + \min_{\symbf{d} \in \mathscr{D}} G_1(\symbf{f}(\symbf{w}), \symbf{d}) \right] + \varepsilon.
		\end{aligned}
	\]
	Since \(\varepsilon\) can be arbitrarily small, this proves the inequality.
\end{proof}

\section*{Acknowledgments}
This paper was supported by the National Natural Science Foundation of China (Grant No. 12372047, No. 62227901).

\section*{Codes and Data Availability Statement} 

The codes and data associated with this paper are available on GitHub at \url{https://github.com/zhong-zh15/Multi\_Flyby\_Dynamic\_Programming}.

\bibliography{sample_EDITED}

\end{document}